\documentclass[10pt]{amsart}

\usepackage[utf8]{inputenc}
\usepackage[T1]{fontenc}
\usepackage[UKenglish]{babel}

\usepackage{amsmath}
\usepackage{amsfonts}
\usepackage{amssymb}
\usepackage{amsthm}
\usepackage{fullpage}
\usepackage{enumerate}
\usepackage{hyperref}
\usepackage{cancel}
\usepackage{geometry}
\usepackage{float}
\usepackage{tikz}
\usetikzlibrary{arrows,matrix,patterns,decorations.markings}
\usepackage{rotating}
\usepackage{lscape}
\usepackage[all,knot]{xy}

\usepackage{mathrsfs}
\usepackage{wasysym}
\usepackage{enumerate}

\theoremstyle{plain}
	\newtheorem{theorem}{\noindent\textbf{Theorem}}
	\newtheorem*{theorem*}{\noindent\textbf{Theorem}}
	\newtheorem{lemma}[theorem]{\noindent\textbf{Lemma}}
	\newtheorem{proposition}[theorem]{\noindent\textbf{Proposition}}
	\newtheorem{corollary}[theorem]{\noindent\textbf{Corollary}}

\theoremstyle{definition}
	\newtheorem{definition*}{\noindent\textbf{Definition}}

\theoremstyle{remark}
	\newtheorem{rem*}{\noindent\textbf{Remark}}

\theoremstyle{definition}
	\newtheorem{definition}{\noindent\textbf{Definition}}
	\newtheorem{example}{\noindent\textbf{Example}}
	\newtheorem{question}{\noindent\textbf{Question}}

\newcommand{\ocross}[2]{\begin{scope}[shift={(#1+2,#2+2)}]
\draw (+2,-2) -- (-2,+2);
\pgfsetlinewidth{8*\pgflinewidth}
\draw[white] (-2,-2) -- (+2,+2);
\pgfsetlinewidth{.125*\pgflinewidth}
\draw (-2,-2) -- (+2,+2);
\end{scope}}

\newcommand{\icross}[2]{\begin{scope}[shift={(#1+2,#2+2)}]
\draw (-2,-2) -- (+2,+2);
\pgfsetlinewidth{8*\pgflinewidth}
\draw[white] (+2,-2) -- (-2,+2);
\pgfsetlinewidth{.125*\pgflinewidth}
\draw (+2,-2) -- (-2,+2);
\end{scope}}

\newcommand{\ocrosstextp}{\raisebox{-0.1cm}{\begin{tikzpicture}[scale=.07]
\draw[<-,thick] (+2,-2) -- (-2,+2);
\pgfsetlinewidth{8*\pgflinewidth}
\draw[white] (-2,-2) -- (+2,+2);
\pgfsetlinewidth{.125*\pgflinewidth}
\draw[->, thick] (-2,-2) -- (+2,+2);
\end{tikzpicture}}\:}

\newcommand{\ocrosstextn}{\raisebox{-0.1cm}{\begin{tikzpicture}[scale=.07]
\draw[->,thick](-2,-2) -- (+2,+2);
\pgfsetlinewidth{8*\pgflinewidth}
\draw[white] (+2,-2) -- (-2,+2);
\pgfsetlinewidth{.125*\pgflinewidth}
\draw[<-,thick]  (+2,-2) -- (-2,+2);
\end{tikzpicture}}\:}

\newcommand{\orisplittext}{\raisebox{-0.1cm}{\begin{tikzpicture}[scale=.07]
\begin{scope}[shift={(2,2)}]
\draw (-2,-2) .. controls +(1,1) and +(-1,1) ..  (2,-2);
\draw (-2,+2) .. controls +(1,-1) and +(-1,-1) ..  (2,+2);
\end{scope}
\draw[dashed] (2,2) circle (2.8cm);
\end{tikzpicture}}\:}

\title{Transverse link invariants from the deformations of Khovanov $\mathfrak{sl}_{3}$-homology}

\author{Carlo Collari}
\begin{document}
\begin{abstract}In this paper we will make use of the Mackaay-Vaz approach to the universal $\mathfrak{sl}_3$-homology to define a family of cycles (called $\beta_3$-invariants) which are transverse braid invariants. This family includes Wu's $\psi_{3}$-invariant. Furthermore, we analyse the vanishing of the homology classes of the $\beta_3$-invariants and relate it to the vanishing of Plamenevskaya's and Wu's invariants. Finally, we use the  $\beta_3$-invariants to produce some Bennequin-type inequalities. 
\end{abstract}
\maketitle
\section{Introduction}
\subsection*{Background material and motivations}
Let $(x,y,z)$ be a system of coordinates in $\mathbb{R}^3$. The \emph{symmetric contact structure} on $\mathbb{R}^{3}$ is the plane distribution $\xi_{sym}=ker(dz + xdy - ydx)$. 
A \emph{link} is a smooth embedding of a number of copies of $\mathbb{S}^1$ into $\mathbb{R}^3$, and a \emph{knot} is a link with a single component. The presence of $\xi_{sym}$ allows one to distinguish a special class of links (and knots): those which are nowhere tangent to $\xi_{sym}$ (\emph{transverse links}). 

Two transverse links are \emph{equivalent} (or the same \emph{transverse type}) if they are ambient isotopic through a one-parameter family of transverse links. In particular, equivalent transverse links represent the same link-type (i.e. the ambient isotopy class of a link).

The study of transverse links has played (and still plays) a key role in the study of low-dimensional topology. The aim of the present paper is to present new invariants for transverse links arising from the deformations of Khovanov $\mathfrak{sl}_3$-homology.
Each transverse link inherits a natural orientation from $\xi_{sym}$ (cf. \cite{Etnyre05}). All transverse invariants defined in this paper are defined with respect to this orientation.

In order to define our invariants we need a particular representation of transverse links. It is a result due to D. Bennequin (\cite{Bennequin83}) that each transverse link is equivalent to a closed braid. Another result, due to S. Orevkov and V. Shevchishin (\cite{OrevkovShev03}) and independently N. Wrinkle (\cite{Wrinkle03}), provides a complete set of combinatorial moves relating all braids whose closure represent the same transverse type. We summarise these results in the following theorem, which shall be referred to as the \emph{transverse Markov theorem} in the rest of the paper.

\begin{theorem}[Bennequin \cite{Bennequin83}, Orevkov and Shevchishin \cite{OrevkovShev03}, Wrinkle  \cite{Wrinkle03}]
Any transverse link is transversely isotopic to the closure of a braid (with axis the $z$-axis).
Moreover, two braids represent the same transverse type if and only if they are related by a finite sequence of braid relations, conjugations, positive stabilisations, and positive destabilisations\footnote{Let $B\in B_{m-1}$, the \emph{positive} (resp. \emph{negative}) \emph{stabilisation} of $B\in B_{m-1}$ is the braid $B\sigma_{m}\in B_{m}$ (resp. $B\sigma_{m}^{-1}\in B_{m}$). The destabilisation is just the inverse process: if one considers a braid of the form $A\sigma_{m} B$ (resp. $A\sigma_{m}^{- 1} B$), where $A,\ B\in B_{m-1}$, then its positive (resp. negative) destabilisation is the braid $AB$.}. These moves are called \emph{transverse Markov moves}.
\end{theorem}

\begin{rem*}
Braids are naturally oriented, and their orientation coincides with the orientation of the corresponding transverse link. 
\end{rem*}

\begin{rem*}
By adding the negative stabilisation and destabilisation to the set of transverse Markov moves one recovers the full set of \emph{Markov moves}.
\end{rem*}

Any sequence of Markov moves between braids, naturally translates into a sequence of oriented Reidemeister moves between their closures. In particular, conjugation in the braid group can be seen as a sequence of Reidemeister moves of the second type followed by a planar isotopy, while the braid relations can be seen as either second or third Reidemeister moves. We remark that not all the oriented versions of second and third Reidemeister moves arise in this way; those which can be obtained as composition of Markov moves and braid relations are called \emph{braid-like} or \emph{coherent}. Finally, a positive (resp. negative) stabilisation translates into a positive (resp. negative) first Reidemeister move, as shown in Figure \ref{fig:stabilisationasr1}. 

\begin{figure}[]
\centering
\begin{tikzpicture}[scale =.7]

\begin{scope}[shift = {+(-8,0)}]
\draw (0,2) circle (1.2);
\draw (0,2) circle (1);
\draw (0,2) circle (2);
\draw (0,2) circle (2.1);
\draw[fill, white] (-.5,-.25) rectangle (0.5,1.25);
\draw (-.5,-.25) rectangle (0.5,1.25);
\node (a) at (0,.5) {$B$};
\end{scope}

\node at (-4,2.75) {Stabilisation};
\node at (-4,2) {\huge{$\rightleftarrows$}};
\node at (-4,1.4) {Destabilisation};
\draw (0,2) circle (1.2);
\draw (0,2) circle (2);
\draw (0,2) circle (2.1);
\draw[fill, white] (-.5,-.25) rectangle (0.5,1.25);
\draw (-.5,-.25) rectangle (0.5,1.25);
\node (a) at (0,.5) {$B$};

\draw (0.5,1.2) .. controls +(.75,.5) and +(.1,-.1) .. (.5,2.5);
\draw  (.5,2.5) .. controls +(-.5,.5) and +(.1,.2) .. (-0.6,2.44);
\pgfsetlinewidth{10*\pgflinewidth}
\draw[white]  (.92,2) .. controls +(0,-1) and +(-.5,-1) .. (-0.6,2.44);
\pgfsetlinewidth{.1*\pgflinewidth}
\draw  (.92,2) .. controls +(0,-1) and +(-.5,-1) .. (-0.6,2.44);

\draw  (.92,2) .. controls +(0,1.25) and +(0,1.25) .. (-.92,2) .. controls +(0,-.25) and +(-.25,.25) .. (-.5,1.2);

\end{tikzpicture}
\caption{Stabilisation and destabilisation as Reidemeister moves between closures.}\label{fig:stabilisationasr1}
\end{figure}
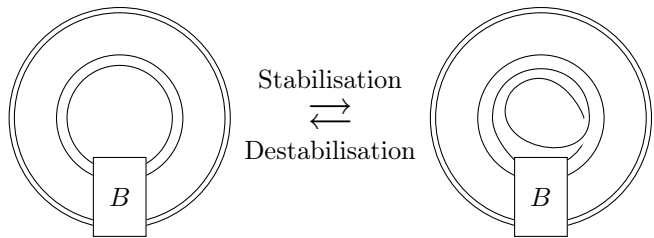

Transverse links have two \emph{classical invariants}: the link-type and the \emph{self-linking number}\footnote{Technically speaking, for links one may record the self-linking number of each components and this defines a slightly stronger transverse invariant, called the self-linking matrix. However, for consticency with the literature in the subject (cf. \cite{TransFromKhType17, Nglipsar13, Plamenevskaya06, Wu08}) in this paper we shall consider only the self-linking number.} $sl$. The latter is defined as follows. Let $B$ be a braid, and set
\[ sl(B) = w(B) - b(B), \]
where $w$ denotes the \emph{writhe}, and $b$ denotes the \emph{braid index} (i.e. the number of strands).
In the light of the transverse Markov theorem, it is immediate that the self-linking number is a transverse invariant. Moreover, from the definition of $sl$ follows easily that a negative stabilisation (resp. destabilisation) does not preserve the equivalence class of the transverse link. Any invariant capable of distinguishing distinct transverse links with the same classical invariants is called \emph{effective}. It is a major problem in the study of transverse links to find effective invariants.

Making use of the representation of transverse links as closed braids, O. Plamenevskaya introduced a transverse invariant $\psi$. This invariant is an homology class in Khovanov's categorification of the Jones polynomial (see \cite{Plamenevskaya06}). Since Plamenevskaya's ground breaking work, other invariants for transverse links coming from quantum (and also Floer theoretic) link homologies have been introduced. In 2008, H. Wu (see \cite{Wu08}) generalised $\psi$ to a family of invariants $\psi_{N}$ which are homology classes in M. Khovanov and L. Rozansky's categorification of the Reshetikhin-Turaev $\mathfrak{sl}_{N}$-invariants.
Khovanov-Rozansky homologies (and also Khovanov homology, being the case $N=2$) admit deformations; these ``deformed theories'' are parametrised by a monic polynomial of degree $N$ called the potential.
In 2015, R. Lipshitz, L. Ng, and S. Sarkar (see \cite{Nglipsar13}) extended the definition of the Plamenevskaya invariant to two chains $\psi^\pm$ belonging to the chain complex of (a twisted version of) Lee's deformation (i.e. the theory associated to the potential $x^{2} - x$). The author in \cite{TransFromKhType17} extended the Plamenevskaya invariant to all deformations of Khovanov homology. It is still unknown at the time of writing whether or not all these invariants are effective.

\subsection*{Outline and statement of results}

The aim of this paper is to extend the definition of Wu's $\psi_3$ invariant to all deformations of Khovanov\footnote{The $\mathfrak{sl}_3$-homology was defined first by Khovanov, then Khovanov and Rozansky extended the definition to all $\mathfrak{sl}_N$, for $N\geq 2$, with a different technique.} $\mathfrak{sl}_3$-homology. As we stated above, these deformations are parametrised by a monic polynomial $\omega\in R[x]$ of degree $3$ (the \emph{potential}), where R is the coefficient ring of the homology theory.
Furthermore, the resulting homology theory is either graded or filtered depending on $\omega$. We make use of the construction of the universal $\mathfrak{sl}_3$-homology due to M. Mackaay and P. Vaz (\cite{Mackaayvaz07}) which encodes all the deformations of Khovanov $\mathfrak{sl}_3$-homology (\cite{Mackaayvaz07a}). We tried to keep the paper as self-contained as possible, so we shall review Mackaay and Vaz's construction in Sections 2 and 3.

In Section 4 we define for each diagram $D$, potential $\omega$, and each root $x_1$ of $\omega$ in $R$, a chain $\beta_{\omega,\: x_1}(D)$, which turns out to be a cycle (Proposition \ref{proposition:betaiscycle}).
If $D= \widehat{B}$ is the closure of a braid $B$, we shall denote $\beta_{\omega,\: x_1}(\widehat{B})$ simply by $\beta_{\omega,\: x_1}(B)$.
Our main theorem is the following.

\begin{theorem}\label{theorem:beta_3}
Let $B$ be a braid.
For each root $x_1$ of $\omega (x)$, the cycle $\beta_{\omega,\: x_1}(\overline{B})$ is a transverse invariant, where the over-line indicates the mirror braid. More precisely, if $B$ and $B^\prime$ are related by a sequence of transverse Markov moves, and denoted by $\Phi$ the map associated to the mirror sequence of these moves, then
\[ \Phi(\beta_{\omega,\: x_1}(\overline{B}))=\beta_{\omega,\: x_1}(\overline{B^\prime}).\]
Moreover, if the homology theory is filtered (resp. graded), then the filtered degree\footnote{That is the index corresponding to the smallest sub-complex in the filtration containing the chain $\beta_{\omega,\: x_1}(\overline{B})$.} (resp. the degree) of $\beta_{\omega,\: x_1}(\overline{B})$ is $-2sl(B)$.
\end{theorem}

\begin{rem*}
The filtered degree mentioned in Theorem \ref{theorem:beta_3} is the filtered degree of \emph{the chain} $\beta_{\omega,\: x_1}(\overline{B})$, and not the filtered degree of its homology class. The filtered degree of $[\beta_{\omega,\: x_1}(\overline{B})]$ is not necessarily a meaningful transverse invariant. For instance, if $\mathbb{F}$ is a field and $\omega\in \mathbb{F}[x]$ has three distinct roots, then the filtered degree of $[\beta_{\omega,\: x_1}(\overline{B})]$ is a concordance invariant (one of the $j_{i}$'s mentioned in Proposition \ref{prop:Bennequin-type}). The comparison between the two filtered degrees, under the above hypotheses, gives a weaker version of the first Bennequin-type inequality in Proposition \ref{prop:Bennequin-type}.
\end{rem*}

Motivated by the theorem above, the cycles $\beta_{\omega,\: x_1}(\overline{B})$ are collectively called $\beta_3$-invariants.
We wish to point out that the construction of the $\beta_3$-invariants seems to be generalizable to the $\mathfrak{sl}_{N}$ case (also with colours). This will be the subject of a forthcoming paper by the author joint with P. Wedrich.

Exploring the effectiveness of these invariants is quite a difficult problem. First, because distinguishing non-equivalent transverse knots with the same classical invariants is a subtle problem. This is also confirmed by the fact that there are only a few known families of transverse knots which are not distinguished by their classical invariants. Second, because of the nature of our invariants: these are chains in a diagram dependent chain complex up to the action of a certain group of chain homotopies. So we leave open the following question.

\begin{question}
Are the $\beta_3$-invariants effective?
\end{question}

In Section 5 we make use of the homology classes of the $\beta_3$-invariants to define two auxiliary invariants which are easier to analyse. These invariants are: the vanishing of the homology class, and the divisibility with respect to a non-unit element of the homology class.

In the case where the base ring is a field we can relate the vanishing of the homology class of the $\beta_3$-invariants to the vanishing of the Plamenevskaya invariant and Wu's $\psi_{3}$-invariant. Our second main result is the following.

\begin{proposition}\label{prop:vanishing}
Let $\omega$ be a potential over a field $\mathbb{F}$, and let $x_1$, $x_2$, $x_3\in \mathbb{F}$ be the roots of $\omega$. Denote by $H_{\omega}^\bullet$ the homology corresponding to the potential $\omega$. Then, given a braid $B$ we have the following:
\begin{itemize}
\item[(1)] if $x_1$ is a simple root of $\omega$, then $[\beta_{\omega, x_1}(\overline{B})]$ is non-trivial in $H_{\omega}^\bullet(\overline{B}, \mathbb{F})$;
\item[(2)] if $x_1$ is a double root of $\omega$, then $[\beta_{\omega, x_1}(\overline{B})]$ vanishes if and only if the Plamenevskaya invariant $\psi(B)$ vanishes in $Kh^{\bullet}(B, \mathbb{F})$;
\item[(3)] if $x_1$ is a triple root of $\omega$, then $[\beta_{\omega, x_1}(\overline{B})]$ vanishes if and only if Wu's invariant $\psi_{3}(B)$ vanishes in $H_{x^{3}}^{\bullet}(\overline{B}, \mathbb{F})$;
\end{itemize}
In particular, the vanishing of $[\beta_{\omega, x_1}(\overline{B})]$ does not depend on the potential $\omega$ and on the root $x_1$, but only on the multiplicity of $x_1$ as a root of $\omega$. (However, it may depend on the choice of the base field.)
\end{proposition}

The previous proposition holds true only if the base ring is a field. This naturally leads to the following question.

\begin{question}
Assume the base ring $R$ to be a Noetherian domain. Is it true that the vanishing of $[\beta_{\omega, x_1}(\overline{B})]$ depends only on the multiplicity of the $x_1$ and on the vanishing of the $\psi$ and $\psi_3$-invariants?
\end{question}

The previous proposition allows us to relate the effectiveness of the vanishing of the homology classes of the $\beta_3$-invariants with the effectiveness of the vanishing of the $\psi$ and $\psi_3$-invariants. In particular, taking into account the results in \cite{TransFromKhType17, Nglipsar13}, we obtain the non-effectiveness of the vanishing of the homology classes of the $\beta_3$-invariants corresponding to double roots in the following cases: knots with crossing number $\leq 11$, flypes, and two bridge knots. Most of the examples of distinct transverse knots with the same classical invariants fall into these categories. However, the following question remains open.

\begin{question}
Is the vanishing of the homology classes of the $\beta_3$-invariants an effective invariant?
\end{question}

Let $R$ be an integral domain and $a\in R\setminus\{ 0\}$ a non-unit element. Given a potential $\omega\in R[x]$ and one of its roots $x_1\in R$, the number
\[c_{\omega, x_1}(B;a) = max\left\{ k\:\vert\: \exists\: [y]\in H^{0}_{\omega}(\overline{B})\:\text{such that}\: a^k[y] = [\beta_{\omega, x_1}(\overline{B})]\right\}\in \mathbb{N}\cup \{ \infty \}\]
is a well-defined transverse invariant, where $c_{\omega,x_1}(B;a)=\infty$ if and only if $[\beta_{\omega, x_1}(\overline{B})]$ is trivial or $a$-torsion. We analyse these invariants in the case $R = \mathbb{F}[U]$ and $\omega = (x - Ux_1) (x - Ux_2) (x - Ux_3)$, with $x_1$, $x_2$, $x_3 \in \mathbb{F}$ distinct, and $a = U$. From our analysis we obtain the following Bennequin-type inequalites.

\begin{proposition}\label{prop:Bennequin-type}
Let $B$ be a braid representing a knot $K$. Given $\omega$ as above, denote by $\omega_{1} = \omega_{\vert U = 1}$. Then the following inequalities hold
\begin{eqnarray}
2 \left[ sl(B) + c_{\omega,Ux_{i}}(B,U) \right] \leq j_{1}(K) \label{eq:inequality1}\\
2 \left[ sl(B) + c_{\omega,Ux_{i}}(B,U)\right] \leq s_{\omega_1}(K) \label{eq:inequality2} \\
sl(B) + c_{\omega,Ux_{i}}(B,U) - 1 \leq 3 \tilde{s}_{\omega_1,x_{i}}(K) \label{eq:inequality3}
\end{eqnarray}
for each $i\in \{ 1,2,3 \}$, where $j_{1}$ (\cite{Lobb12, Wu09}, see also \cite{LewarkLobb17}), $s_{\omega_1}(K)$, and $\tilde{s}_{\omega_1,x_{i}}(K)$ (\cite{LewarkLobb17}) are concordance invariants (see also Theorem \ref{thm:jis} and subsequent lines for a very quick overview).
\end{proposition}

\subsection*{Acknowledgments}
The author wish to thank Prof. Paolo Lisca for his advice, the helpful conversations and his continuous support. Moreover, the author also wishes to thank M. Mackaay and P. Vaz for letting him use the images in Figures \ref{fig:secondRmcoer} and \ref{fig:invterzasl3a}. 
Finally, the author wishes to thank the referees for their valuable comments and suggestions.
This paper is partially excerpted from the author's PhD thesis. During his PhD the author was supported by a PhD scholarship ``Firenze-Perugia-Indam''.

\section{Webs and Foams}\label{sec:webetfoams}
In this section we will briefly review the definition of webs and foams. 
These objects play the same role played by closed $1$-dimensional manifolds and surfaces in Bar-Natan geometric description of Khovanov homology (and its deformations).

\subsection{Webs}
Webs were originally introduced by Greg Kuperberg in \cite{Kuperberg96}, as a tool to study the representation theory of rank $2$ Lie algebras, and used by Khovanov in \cite{Khovanov03} to define a categorification of the $\mathfrak{sl}_3$-Jones polynomial.

\begin{definition}
A \emph{web} $W$ is a directed trivalent planar graph embedded in $\mathbb{R}^{2}$, possibly with components without vertices (\emph{loops}), satisfying the following properties:
\begin{enumerate}[(a)]
\item $W$ has a finite number of vertices and a finite number of loops;
\item there are two types of edges in $W$, the \emph{thin edges} and the \emph{thick edges}, and for each vertex $v\in V(W)$ there is a unique thick edge incident in $v$;
\item each vertex of $W$ is either a source\footnote{A vertex $v$ of a directed graph is a \emph{source} if all edges incident in $v$ are directed outwards from $v$.} or a sink\footnote{A vertex $v$ of a directed graph is a \emph{sink} if all edges incident in $v$ are directed towards $v$.}.
\end{enumerate}
For technical reasons also the empty set is considered a web (the \emph{empty web}). A trivalent directed (abstract) graph $\Gamma$ satisfying (a) and (b) shall be called \emph{abstract web}.
\end{definition}

The distinction between thick and thin edges is necessary to keep track of crossings after their resolution (cf. Subsection \ref{subsection:homologies}). Whenever necessary the thick edges shall be drawn thicker and coloured purple, otherwise no distinction shall be made. 

\subsection{Foams} Roughly speaking, foams are decorated branched surfaces which are singular along a smooth $1$-dimensional manifold of triple points. Let us put aside the decorations and let us start by defining the underlying topological structure of a foam.

A (\emph{topological}) \emph{pre-foam} $\Sigma$ is a compact topological space such that each point has a neighbourhood homeomorphic to one of the four local models in Figure \ref{fig:Localmodelsingcurvefoam}. 
A point $p\in \Sigma$ is called \emph{regular} if it has a neighbourhood which is homeomorphic to either (C) or (D). Non-regular points are called \emph{singular}, and the set of singular points is denoted by $Sing(\Sigma)$. 
The connected components of $\Sigma \setminus Sing(\Sigma) \subseteq \Sigma$ are called \emph{regular regions} of $\Sigma$.
Finally, a \emph{boundary} point for $\Sigma$ is a point which does not have a neighbourhood homeomorphic to either (A) or (C). A topological pre-foam
with empty boundary is called \emph{closed}.
\begin{figure}[H]
\centering
\begin{tikzpicture}
\draw[dashed] ( .866,.5) -- +(-.25,-.5) -- (-1.116,-1)--( -.866,-.5) ;
\draw[fill, color = white, opacity = .75] ( -.866,-.5) -- (-.116,-1) -- (1.614,0) -- (.866,.5) -- cycle;
\draw[dashed] ( -.866,-.5) -- (-.116,-1) -- (1.614,0) -- (.866,.5) -- cycle;
%\draw[fill] (0,0) circle (.05 cm); 
%\node[above] at (0,0) {$p$};
\draw[dashed] ( -.866,-.5) -- (-.866,.5) -- ( .866,1.5) -- (.866,.5);
\draw[thick, red] ( -.866,-.5) --  (.866,.5);
\draw[->, red] ( 1.5,1) .. controls +(-.25 , 0) and +(.25,.25) ..  (.5,.5);
\node at (2.75,1) {\textcolor{red}{singular points}};
\begin{scope}[shift ={+(6,0)}]
\draw[dashed] ( .866,.5) -- +(-.25,-.5) -- (-1.116,-1)--( -.866,-.5) ;
\draw[very thick] ( -.866,-.5) -- (-1.116,-1);
\draw[fill, color = white, opacity = .75] ( -.866,-.5) -- (-.116,-1) -- (1.614,0) -- (.866,.5) -- cycle;
\draw[dashed] ( -.866,-.5) -- (-.116,-1) -- (1.614,0) -- (.866,.5) -- cycle;
\draw[very thick] ( -.866,-.5) -- (-.116,-1);
%\draw[fill] (0,0) circle (.05 cm); 
%\node[above] at (0,0) {$p$};
\draw[dashed] ( -.866,-.5) -- (-.866,.5) -- ( .866,1.5) -- (.866,.5)-- cycle;
\draw[very thick ] ( -.866,-.5) -- (-.866,.5);
\draw[thick, red] ( -.866,-.5) --  (.866,.5);
\end{scope}
\draw[->, red] (4,1) .. controls +(.25 , 0) and +(-.25,.25) ..  (6.1,.25);

\draw[dashed] (-.9,-2) rectangle (1.6,-3);

\draw[dashed] (7.6,-3) -- (7.6,-2) -- (5.1,-2) -- (5.1,-3);
\draw[very thick] (7.6,-3)  -- (5.1,-3);

\node at (0.25,-1.5) {\small{(A)}};
\node at (6.25,-1.5) {\small{(B)}};
\node at (0.25,-3.5) {\small{(C)}};
\node at (6.25,-3.5) {\small{(D)}};

\end{tikzpicture}
\caption{Local models for a pre-foam.}
\label{fig:Localmodelsingcurvefoam}
\end{figure}
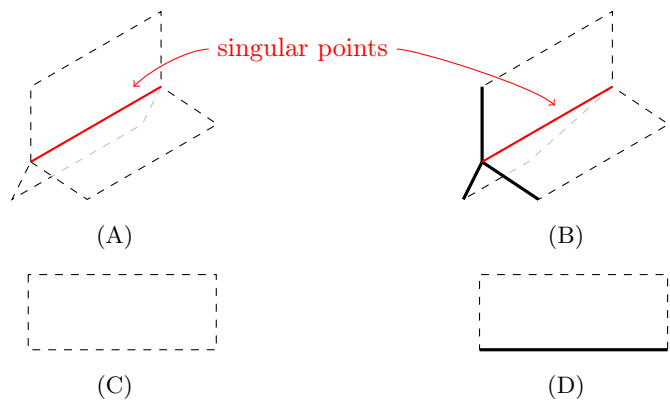

\begin{rem*}
The singular locus of  pre-foam is the disjoint union of circles and arcs, which are called \emph{singular circles} and \emph{singular arcs}. The singular boundary points correspond to the boundary points of the singular arcs.
\end{rem*}

The choice of an atlas\footnote{We mean an open cover of $\Sigma$ together with a homeomorphism of each element of the cover with one of the local models in Figure \ref{fig:Localmodelsingcurvefoam}.} on a pre-foam determines an atlas on each regular region and also on $Sing(\Sigma)$. 
A \emph{smooth pre-foam} is a topological pre-foam with a chosen topological atlas such that the induced atlases on the regular regions and on the singular locus are smooth atlases. 
Similarly, given a pre-foam $\Sigma$, an \emph{orientation} on $\Sigma$ is the choice of an orientation of the closure of each regular region in such a way that the orientation induced in the intersection of two closed regions agrees.

\begin{rem*}
If we choose an orientation on a orientable pre-foam, this induces an orientation on the singular locus.
\end{rem*}

\begin{rem*}
The boundary of a topological pre-foam is a (possibly empty) finite trivalent graph whose vertices correspond to singular boundary points. If the pre-foam is oriented its boundary is a directed graph. Moreover, each vertex of the boundary graph of an oriented pre-foam is either a sink or a source. In other words, the boundary of an oriented pre-foam is an abstract web.
\end{rem*}

\begin{definition}
A \emph{decorated pre-foam} is an oriented pre-foam $\Sigma$ together with the following data:
\begin{enumerate}[(a)]
\item a finite number (possibly zero) of marked points, called \emph{dots}, in the interior of each regular region of $\Sigma$;
\item a cyclic order on the regular regions incident to a singular arc or circle.
\end{enumerate}
\end{definition}

It is possible to define a category $\mathbf{PreFoam}$ whose objects are abstract webs and whose morphisms are formal $R$-linear combinations of the triples $(\Sigma,\partial_{0}\Sigma ,\partial_{1}\Sigma)$ satisfying the following properties
\begin{itemize}
\item[$\triangleright$] $\Sigma$ is a decorated pre-foam;
\item[$\triangleright$] $\partial_{0}\Sigma,\: \partial_{1}\Sigma \in Obj(\mathbf{PreFoam})$, $\partial_{0}\Sigma$ is the source object and $\partial_{1} \Sigma$ the target object of the morphism;
\item[$\triangleright$] $\partial\Sigma = \partial_{0} \Sigma \sqcup -\partial_1 \Sigma$, where the minus sign denotes the reversal of the orientation;
\item[$\triangleright$] the triple is seen up to boundary fixing isotopies which do not change the regular regions of the dots and preserve the ordering of the components near each singular arc.
\end{itemize}
Finally, the composition of two triples $(\Sigma,\partial_{0}\Sigma ,\partial_{1}\Sigma)$ and $(\Sigma^\prime,\partial_{1}\Sigma ,\partial_{1}\Sigma^\prime)$ is defined as the triple  $(\Sigma^{\prime\prime},\partial_{0}\Sigma ,\partial_{1}\Sigma^{\prime})$, where $\Sigma^{\prime\prime}$ is obtained by glueing $\Sigma$ and $\Sigma^\prime$ along $\partial_{1}\Sigma$.
\begin{definition}
A \emph{foam} is a decorated pre-foam properly and smoothly\footnote{That is in such a way that the restriction of the embedding to each regular region and to the singular locus is smooth.} embedded in $\mathbb{R}^{2} \times I$. Moreover, we ask the cyclic order of the regular regions at each singular arc or circle to coincide with the cyclic order induced by rotating clockwise\footnote{We suppose fixed an orientation of $\mathbb{R}^{2}\times I$.} around a singular arc or circle. Foams will be considered up to ambient isotopies of $\mathbb{R}^{2}\times I$ which fix the boundary of the foam and do not change the regular regions of the dots.
\end{definition}

Given two webs $W_{0}$ and $W_{1}$, a \emph{foam between $W_{0}$ and $W_{1}$} is a foam $F$ such that 
\[ F  \cap \mathbb{R}^{2} \times \{ 0 \} = W_{0}\quad \text{and}\quad F  \cap \mathbb{R}^{2} \times \{ 1 \} = -W_{1}.\]
The category $\mathbf{Foam}$ is the category whose objects are webs, whose morphisms are a $R$-linear combinations of foams between two webs, and whose composition is defined as the glueing of two foams along the shared boundary. For technical reasons also the empty foam is included in $\mathbf{Foam}$ as an element of $Hom_{\mathbf{Foam}}(\emptyset,\emptyset)$.

\subsection{Local relations}

Local relations are equalities among (linear combinations of) foams which are identical except inside a small ball. The relations we are concerned with can be divided into two types: 
\begin{itemize}
\item[$\triangleright$] \emph{reduction relations} (c.f. Figure \ref{fig:localrel1}), 
\item[$\triangleright$] \emph{evaluation relations} (c.f. Figure \ref{fig:localrel2}).
\end{itemize}
The reduction relations depend on the choice of a polynomial $\omega(x)\in R [x]$ (the \emph{potential}) of the form
\[ \omega(x) = x^{3} + a_2x^{2} + a_1x + a_0. \]
Thus, we will hereby suppose $\omega$ fixed. In the case $R$ is a graded ring, one may require the coefficients $a_2$, $a_1$ and $a_0$ and the polynomial $\omega$ to be homogeneous in order to obtain a graded theory. Let us get back to the local relations and postpone the matter of the gradings. 

The reduction relations allow one to reduce either the number of handles (\emph{genus reduction relation} (GR)) or the number of dots  (\emph{dot reduction relation} (DR)) at the expense of trading a single foam for a linear combination of foams. 

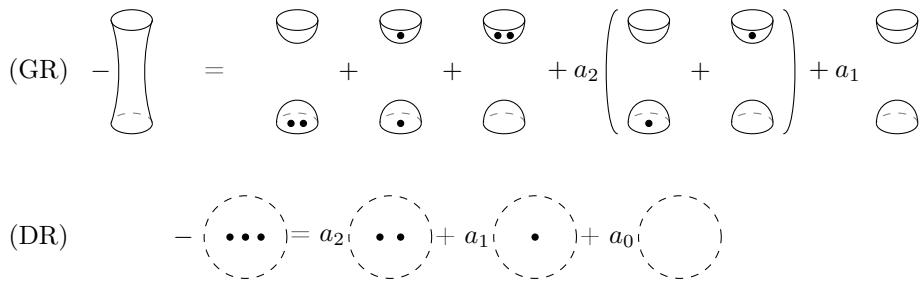
\begin{figure}[H]
\centering
\begin{tikzpicture}[scale =.275]

\begin{scope}[shift = {(-2,0)}]
\draw (-3.5,18)circle (  1 and .5);
\draw (-2.5,13) arc (0:-180:1 and .5);
\draw[dashed, gray] (-2.5,13) arc (0:180:1 and .5);
\draw (-2.5,13) .. controls +(-.5,1.5) and +(-.5,-1.5) .. (-2.5,18);
\draw (-4.5,13) .. controls +(.5,1.5) and +(.5,-1.5) .. (-4.5,18);

\draw (4.5,18)circle ( 1 and .5);
\draw (5.5,13) arc (0:-180:1 and .5);
\draw[dashed, gray] (5.5,13) arc (0:180:1 and .5);
\draw (5.5,13) .. controls +(0,1.5) and +(0,1.5) .. (3.5,13);
\draw (5.5,18) .. controls +(0,-1.5) and +(0,-1.5) .. (3.5,18);

\draw (9.5,18)circle ( 1 and .5);
\draw (10.5,13) arc (0:-180:1 and .5);
\draw[dashed,gray] (10.5,13) arc (0:180:1 and .5);
\draw (10.5,13) .. controls +(0,1.5) and +(0,1.5) .. (8.5,13);
\draw (10.5,18) .. controls +(0,-1.5) and +(0,-1.5) .. (8.5,18);

\draw (14.5,18)circle ( 1 and .5);
\draw (15.5,13) arc (0:-180:1 and .5);
\draw[dashed,gray] (15.5,13) arc (0:180:1 and .5);
\draw (15.5,13) .. controls +(0,1.5) and +(0,1.5) .. (13.5,13);
\draw (15.5,18) .. controls +(0,-1.5) and +(0,-1.5) .. (13.5,18);

\draw (21.5,18)circle ( 1 and .5);
\draw (22.5,13) arc (0:-180:1 and .5);
\draw[dashed,gray] (22.5,13) arc (0:180:1 and .5);
\draw (22.5,13) .. controls +(0,1.5) and +(0,1.5) .. (20.5,13);
\draw (22.5,18) .. controls +(0,-1.5) and +(0,-1.5) .. (20.5,18);

\draw (26.5,18)circle ( 1 and .5);
\draw (27.5,13) arc (0:-180:1 and .5);
\draw[dashed,gray] (27.5,13) arc (0:180:1 and .5);
\draw (27.5,13) .. controls +(0,1.5) and +(0,1.5) .. (25.5,13);
\draw (27.5,18) .. controls +(0,-1.5) and +(0,-1.5) .. (25.5,18);

\draw (33.5,18)circle ( 1 and .5);
\draw (34.5,13) arc (0:-180:1 and .5);
\draw[dashed,gray] (34.5,13) arc (0:180:1 and .5);
\draw (34.5,13) .. controls +(0,1.5) and +(0,1.5) .. (32.5,13);
\draw (34.5,18) .. controls +(0,-1.5) and +(0,-1.5) .. (32.5,18);

\node at (-8,15.5)  {(GR)};
\node at (0.5,15.5)  {=};
\node at (-5,15.5)  {$-$};
\node at (7 ,15.5)  {$+$};
\node at (12 ,15.5)  {$+$};
\node at (18 ,15.5)  {$+\: a_2\:$};
\draw (20,18.5) .. controls +(-.75,0) and +(-.75,0) .. (20,12.5);
\draw (28,18.5) .. controls +(.75,0) and +(.75,0) .. (28,12.5);
\node at (24 ,15.5)  {$+$};
\node at (30.5 ,15.5)  {$+\: a_1$};

\draw[fill] (4.8,13) circle (0.15) ;
\draw[fill] (4.2,13) circle (0.15) ;

\draw[fill] (9.5,13) circle (0.15) ;
\draw[fill] (9.5,17.25) circle (0.15) ;

\draw[fill] (14.8,17.25) circle (0.15) ;
\draw[fill] (14.2,17.25) circle (0.15) ;

\draw[fill] (21.5,13) circle (0.15) ;
\draw[fill] (26.5,17.25) circle (0.15) ;

\end{scope}

\draw[dashed] (0,7.5) circle (2);
\draw[dashed] (7,7.5) circle (2);
\draw[dashed] (14,7.5) circle (2);
\draw[dashed] (21,7.5) circle (2);

\node at (3.5,7.5)  {= $a_2$};
\node at (10.5,7.5)  {+ $a_1$};
\node at (17.5,7.5)  {+ $a_0$};
\node at (-3,7.5)  {$-$};

\node at (-10,7.5)  {(DR)};

\draw[fill] (7.5,7.5) circle (0.15) ;
\draw[fill] (6.5,7.5) circle (0.15) ;
\draw[fill] (0,7.5) circle (0.15) ;
\draw[fill] (0.75,7.5) circle (0.15) ;
\draw[fill] (-.75,7.5) circle (0.15) ;
\draw[fill] (14,7.5) circle (0.15) ;
\end{tikzpicture}
\caption{Reduction relations.}
\label{fig:localrel1}
\end{figure}

There are two types of evaluation relations. The first type are called \emph{sphere relations} (S), and concern spheres with less than $4$ dots.  The second type are called \emph{theta foam relations} ($\mathrm{\Theta}$), and concern theta foams (that is spheres with a disk glued along the equator) with $2$ dots or less on each region.

\begin{figure}[H]
\centering
\begin{tikzpicture}[scale = .3]
\draw (0,0) circle (2);
\draw (7,0) circle (2);
\draw (19,0) circle (2);

\draw[gray,dashed] (-2,0) arc (-180:0: 2 and 1);
\draw[gray,dashed,opacity = .5] (-2,0) arc (180:0: 2 and 1);

\draw[gray,dashed] (5,0) arc (-180:0: 2 and 1);
\draw[gray,dashed,opacity = .5] (5,0) arc (180:0: 2 and 1);

\draw[gray,dashed] (17,0) arc (-180:0: 2 and 1);
\draw[gray,dashed,opacity = .5] (17,0) arc (180:0: 2 and 1);

\draw[fill] (7,-.75) circle (0.15) ;
\draw[fill] (19.5,-.75) circle (0.15) ;
\draw[fill] (18.5,-.75) circle (0.15) ;

\node at (3.5,0)  {=};
\node at (10.5,0)  {=};
\node at (12.5,0)  {0};

\node at (22.5,0)  {=};
\node at (24,0)  {-1};
\node at (-5,0)  {(S)};

\node at (-5,-7)  {($\Theta$)};
\draw (0,-7) circle (2);
\draw[pattern color = red, pattern= north west lines, opacity = .45] (0,-7) circle (2 and 1);
\draw[thick,red] (0,-8) arc (-90:-180: 2 and 1);
\draw[->,thick,red] (2,-7) arc (0:-90: 2 and 1);
\draw[dashed,red] (-2,-7) arc (180:0: 2 and 1);

\draw (8.5,-7) circle (2);
\draw[pattern color = red, pattern= north west lines, fill, opacity = .45] (8.5,-7) circle (2 and 1);
\draw[<-,thick,red] (8.5,-8) arc (-90:-180: 2 and 1);
\draw[thick,red] (10.5,-7) arc (0:-90: 2 and 1);
\draw[dashed,red] (6.5,-7) arc (180:0: 2 and 1);
\node at (4.5,-7) {$=\:-$};
\node at (26,-7) {\small{$=\begin{cases} 1 & ( n_1, n_2, n_3 ) = (1,2,0)\:\text{or a cyclic permutation},\\ -1 & ( n_1, n_2, n_3 ) = (2,1,0)\:\text{or a cyclic permutation},\\%
0 & n_1,n_2,n_3 \leq 2\text{, and not in the previous cases}\\\end{cases}$}};
\node at (8.5,-5.5) {\small{$n_1$}};
\node at (8.5, -7) {\small{$n_2$}};
\node at (8.5,-8.6) {\small{$n_{3}$}};
\node at (0,-5.5) {\small{$n_1$}};
\node at (0, -7) {\small{$n_2$}};
\node at (0,-8.6) {\small{$n_{3}$}};

\end{tikzpicture}
\caption{Evaluation relations. The numbers $n_1$, $n_2$ and $n_3$ on the theta foam indicate the number of dots in the corresponding region.}
\label{fig:localrel2}
\end{figure}

Modulo the local relations, each closed foam is equivalent to a multiple of the empty foam; using the reduction relations one reduces a the closed foam to a $R$-linear combination of (disjoint unions of) spheres and theta foams with less than three dots in each regular region. Finally, one uses the evaluation relations to obtain a multiple of the empty foam. 
For each pair of webs $W_0$ and $W_1$, and each foam $F\in Hom_{\mathbf{Foam}}(W_0, W_1)$, there is a well defined $R$-bilinear pairing
\[ \langle \cdot \vert \cdot \rangle_F: Hom_{\mathbf{Foam}}(\emptyset, W_0) \otimes Hom_{\mathbf{Foam}}(W_1, \emptyset) \longrightarrow R (= R \langle \emptyset \rangle) \] 
given by ``capping off'' $F$ with an element of $Hom_{\mathbf{Foam}}(\emptyset, W_0)$ and an element of $Hom_{\mathbf{Foam}}(W_1, \emptyset)$, and evaluating it.

Now, we can define the category $\mathbf{Foam}_{/\ell}$, as the category whose object are the same as $\mathbf{Foam}$, but the morphisms are considered to be equal if the corresponding bilinear forms are equal. Using the local relations it is possible to prove the following result. The reader is referred to \cite{Mackaayvaz07} for a proof.
%\newpage
\begin{proposition}[Mackaay-Vaz, \cite{Mackaayvaz07}]\label{proposition:Mackaayvazrel}
The following local relations hold in $\mathbf{Foam}_{/\ell}$
\begin{figure}[H]
\centering
\usetikzlibrary{patterns}
\begin{tikzpicture}[scale = .25]

%\node at (-5,7) {$(4TF)$};
\draw[] (3,8) arc (-90:90:.5 and 1);
\draw (-3,8) arc (-90:90:.5 and 1);
\draw[dashed] (3,10) arc (90:270:.5 and 1);
\draw (-3,10) arc (90:270:.5 and 1);
\draw (-1,5) arc (180:360:1 and .5);
\draw[dashed] (1,5) arc (0:180:1 and .5);

\draw (-3,8) -- (3,8);
\draw (-3,10) -- (3,10);
\draw (-1,5) arc (180:0:1 and 1);

\draw[pattern = north west lines, pattern color = red] (0,9) circle (.5 and 1);
\draw[color =white] (0,9) circle (.5 and 1);
\draw[thick, ->] (0,8) arc (-90:0:.5 and 1); 
\draw[ thick]  (.5,9) arc (0:90:.5 and 1);
\draw[dashed, thick] (0,10) arc (90:270:.5 and 1);

\draw[fill] (0,9) circle (.15);

\node at (15,7) {$=$};

\begin{scope}[shift = {+(20,0)}]
\draw[] (3,8) arc (-90:90:.5 and 1);
\draw (-3,8) arc (-90:90:.5 and 1);
\draw[dashed] (3,10) arc (90:270:.5 and 1);
\draw (-3,10) arc (90:270:.5 and 1);
\draw (-1,5) arc (180:360:1 and .5);
\draw[dashed] (1,5) arc (0:180:1 and .5);

\draw (-3,8) -- (3,8);
\draw (-3,10) -- (3,10);
\draw (-1,5) arc (180:0:1 and 1);

\draw[pattern = north west lines, pattern color = red] (0,9) circle (.5 and 1);
\draw[color =white] (0,9) circle (.5 and 1);
\draw[thick, ->] (0,8) arc (-90:0:.5 and 1); 
\draw[ thick]  (.5,9) arc (0:90:.5 and 1);
\draw[dashed, thick] (0,10) arc (90:270:.5 and 1);

\draw[fill] (0,5) circle (.15);
\end{scope}
\node at (5,7) {$+$};

\begin{scope}[shift = {+(10,0)}]
\draw[] (3,8) arc (-90:90:.5 and 1);
\draw (-3,8) arc (-90:90:.5 and 1);
\draw[dashed] (3,10) arc (90:270:.5 and 1);
\draw (-3,10) arc (90:270:.5 and 1);
\draw (-1,5) arc (180:360:1 and .5);
\draw[dashed] (1,5) arc (0:180:1 and .5);

\draw (1,5) .. controls +(0,1) and +(-1.5,0) .. (3,8);
\draw (-1,5) .. controls +(0,2) and +(-3,0) .. (3,10);
\draw (-3,10) arc (90:-90:1 and 1);

\end{scope}
\node at (25,7) {$+$};

\begin{scope}[shift = {+(30,0)}]
\draw[] (3,8) arc (-90:90:.5 and 1);
\draw (-3,8) arc (-90:90:.5 and 1);
\draw[dashed] (3,10) arc (90:270:.5 and 1);
\draw (-3,10) arc (90:270:.5 and 1);
\draw (-1,5) arc (180:360:1 and .5);
\draw[dashed] (1,5) arc (0:180:1 and .5);

\draw (-1,5) .. controls +(0,1) and +(1.5,0) .. (-3,8);
\draw (1,5) .. controls +(0,2) and +(3,0) .. (-3,10);
\draw (3,10) arc (90:270:1 and 1);

\end{scope}

\end{tikzpicture}

\begin{tikzpicture}[scale = .25]
\draw[white] (-1,6)--(1,6);
\draw (-1,.5)  arc (-180:0:1 and .5);
\draw[dashed] (-1,.5)  arc (180:0:1 and .5);

\draw (0,4)  circle (-1 and .5);
\draw[pattern = north west lines, pattern color = red] (3,4)  - - (1,4) --  (1,.5)  - - (3,.5);
\draw[pattern = north west lines, pattern color = red] (-3,4)  - - (-1,4) --  (-1,.5)  - - (-3,.5);
\draw (3,4)  - - (1,4)  (1,.5)  - - (3,.5);
\draw[thick, <-] (-1,2.25)  - - (-1,.5);
\draw[thick, ->] (1,4)  - - (1,2.25);
\draw[thick] (-1,4)  - - (-1,.5);
\draw[thick] (1,4)  - - (1,.5);
\draw (-3,4)  - - (-1,4)  (-1,.5)  - - (-3,.5);
\draw (3,4)  - - (1,4)  (1,.5)  - - (3,.5);

\node at (5,2.25) {=};
\begin{scope}[shift = {+(10,0)}]
\draw[pattern = north west lines, pattern color = red] (-3,4) -- (-1,4) arc (-180:0:1) -- (3,4) -- (3,.5) -- (1,.5) arc (0:180:1)  -- (-3,.5) -- cycle;
\draw[white, thick] (-3,4) -- (-1,4) arc (-180:0:1) -- (3,4) -- (3,.5) -- (1,.5) arc (0:180:1)  -- (-3,.5) -- cycle;
\draw (-1,.5)  arc (-180:0:1 and .5);
\draw[dashed] (-1,.5)  arc (180:0:1 and .5);
\draw[thick, ->] (-1,.5)  arc (180:0:1);

\draw (0,4)  circle (-1 and .5);
\draw[thick, <-] (-1,4)  arc (-180:0:1);

\draw (-3,4)  - - (-1,4)  (-1,.5)  - - (-3,.5);
\draw (3,4)  - - (1,4)  (1,.5)  - - (3,.5);
\draw[fill] (0,4) circle  (.15);
\end{scope}
\node at (15,2.25) {+};
\begin{scope}[shift = {+(20,0)}]
\draw[pattern = north west lines, pattern color = red] (-3,4) -- (-1,4) arc (-180:0:1) -- (3,4) -- (3,.5) -- (1,.5) arc (0:180:1)  -- (-3,.5) -- cycle;
\draw[white, thick] (-3,4) -- (-1,4) arc (-180:0:1) -- (3,4) -- (3,.5) -- (1,.5) arc (0:180:1)  -- (-3,.5) -- cycle;
\draw (-1,.5)  arc (-180:0:1 and .5);
\draw[dashed] (-1,.5)  arc (180:0:1 and .5);
\draw[thick, ->] (-1,.5)  arc (180:0:1);

\draw (0,4)  circle (-1 and .5);
\draw[thick, <-] (-1,4)  arc (-180:0:1);

\draw (-3,4)  - - (-1,4)  (-1,.5)  - - (-3,.5);
\draw (3,4)  - - (1,4)  (1,.5)  - - (3,.5);
\draw[fill] (0,.5) circle  (.15);
\end{scope}

%\node at (-20,2.25) {$(NCF)$};
\begin{scope}[shift ={+(-18,0)}]
\draw (-1,.5)  arc (-180:0:1 and .5);
\draw[dashed] (-1,.5)  arc (180:0:1 and .5);

\draw (0,4)  circle (-1 and .5);

\draw[pattern = north west lines, pattern color = red] (0,2.25)  circle (1 and .5);
\draw[white, thick] (0,2.25)  circle (1 and .5);
\draw[dashed, thick] (-1,2.25)  arc (180:0:1 and .5);
\draw[thick] (0,1.75)  arc (270:360:1 and .5);
\draw[->, thick] (-1,2.25)  arc (180:270:1 and .5);
\draw[] (-1,4)  - - (-1,.5);
\draw[] (1,4)  - - (1,.5);

\node at (2.5,2.25) {=};
\begin{scope}[shift = {+(5,0)}]

\draw (-1,.5)  arc (-180:0:1 and .5);
\draw[dashed] (-1,.5)  arc (180:0:1 and .5);
\draw[] (-1,.5)  arc (180:0:1);

\draw (0,4)  circle (-1 and .5);
\draw[] (-1,4)  arc (-180:0:1);

\draw[fill] (0,4) circle  (.125);
\end{scope}
\node at (7.5,2.25) {$-$};
\begin{scope}[shift = {+(10,0)}]

\draw (-1,.5)  arc (-180:0:1 and .5);
\draw[dashed] (-1,.5)  arc (180:0:1 and .5);
\draw[] (-1,.5)  arc (180:0:1);

\draw (0,4)  circle (-1 and .5);
\draw[] (-1,4)  arc (-180:0:1);

\draw[fill] (0,.5) circle  (.125);
\end{scope}
\end{scope}
\end{tikzpicture}
\end{figure}
\begin{figure}[H]
\begin{tikzpicture}[scale=.5]
\draw[dashed] ( .866,.5) -- +(-.25,-.75) -- (-1.116,-1.25)--( -.866,-.5) ;
\draw[fill, color = white, opacity = .76] ( -.866,-.5) -- (.116,-.75) -- (1.73,.25) -- (.866,.5) -- cycle;
\draw[dashed] ( -.866,-.5) -- (.116,-.75) -- (1.73,.25) -- (.866,.5) -- cycle;\draw[dashed] ( -.866,-.5) -- (-.866,.5) -- ( .866,1.5) -- (.866,.5);
\draw[thick, red] ( -.866,-.5) --  (.866,.5);
\draw[fill] (-.5,.125) circle (0.075);

\begin{scope}[shift ={+(5,0)}]
\draw[dashed] ( .866,.5) -- (.616,-.25) -- (-1.116,-1.25)--( -.866,-.5) ;
\draw[fill, color = white, opacity = .76] ( -.866,-.5) -- (.116,-.75) -- (1.73,.25) -- (.866,.5) -- cycle;
\draw[dashed] ( -.866,-.5) -- (.116,-.75) -- (1.73,.25) -- (.866,.5) -- cycle;\draw[dashed] ( -.866,-.5) -- (-.866,.5) -- ( .866,1.5) -- (.866,.5);
\draw[thick, red] ( -.866,-.5) --  (.866,.5);
\draw[fill] (0,-.25) circle (0.075);
\end{scope}

\begin{scope}[shift ={+(10,0)}]
\draw[dashed] ( .866,.5) -- (.616,-.25) -- (-1.116,-1.25)--( -.866,-.5) ;
\draw[fill, color = white, opacity = .76] ( -.866,-.5) -- (.116,-.75) -- (1.73,.25) -- (.866,.5) -- cycle;
\draw[dashed] ( -.866,-.5) -- (.116,-.75) -- (1.73,.25) -- (.866,.5) -- cycle;\draw[dashed] ( -.866,-.5) -- (-.866,.5) -- ( .866,1.5) -- (.866,.5);
\draw[thick, red] ( -.866,-.5) --  (.866,.5);
\draw[fill] (-.75,-.75) circle (0.075);
\end{scope}

\begin{scope}[shift ={+(15,0)}]
\draw[dashed] ( .866,.5) -- (.616,-.25) -- (-1.116,-1.25)--( -.866,-.5) ;
\draw[fill, color = white, opacity = .76] ( -.866,-.5) -- (.116,-.75) -- (1.73,.25) -- (.866,.5) -- cycle;
\draw[dashed] ( -.866,-.5) -- (.116,-.75) -- (1.73,.25) -- (.866,.5) -- cycle;\draw[dashed] ( -.866,-.5) -- (-.866,.5) -- ( .866,1.5) -- (.866,.5);
\draw[thick, red] ( -.866,-.5) --  (.866,.5);
\end{scope}

\node  at (2.75,0) {+};
\node  at (7.75,0) {+};
\node  at (13,0) {= $\ -a_2$};

\begin{scope}[shift={(0,-3)}]
\draw[dashed] ( .866,.5) -- +(-.25,-.75) -- (-1.116,-1.25)--( -.866,-.5) ;
\draw[fill, color = white, opacity = .76] ( -.866,-.5) -- (.116,-.75) -- (1.73,.25) -- (.866,.5) -- cycle;
\draw[dashed] ( -.866,-.5) -- (.116,-.75) -- (1.73,.25) -- (.866,.5) -- cycle;\draw[dashed] ( -.866,-.5) -- (-.866,.5) -- ( .866,1.5) -- (.866,.5);
\draw[thick, red] ( -.866,-.5) --  (.866,.5);
\draw[fill] (-.5,.125) circle (0.075);
\draw[fill] (-.75,-.75) circle (0.075);

\begin{scope}[shift ={+(5,0)}]
\draw[dashed] ( .866,.5) -- (.616,-.25) -- (-1.116,-1.25)--( -.866,-.5) ;
\draw[fill, color = white, opacity = .76] ( -.866,-.5) -- (.116,-.75) -- (1.73,.25) -- (.866,.5) -- cycle;
\draw[dashed] ( -.866,-.5) -- (.116,-.75) -- (1.73,.25) -- (.866,.5) -- cycle;\draw[dashed] ( -.866,-.5) -- (-.866,.5) -- ( .866,1.5) -- (.866,.5);
\draw[thick, red] ( -.866,-.5) --  (.866,.5);
\draw[fill] (-.5,.125) circle (0.075);
\draw[fill] (0,-.25) circle (0.075);
\end{scope}

\begin{scope}[shift ={+(10,0)}]
\draw[dashed] ( .866,.5) -- (.616,-.25) -- (-1.116,-1.25)--( -.866,-.5) ;
\draw[fill, color = white, opacity = .76] ( -.866,-.5) -- (.116,-.75) -- (1.73,.25) -- (.866,.5) -- cycle;
\draw[dashed] ( -.866,-.5) -- (.116,-.75) -- (1.73,.25) -- (.866,.5) -- cycle;\draw[dashed] ( -.866,-.5) -- (-.866,.5) -- ( .866,1.5) -- (.866,.5);
\draw[thick, red] ( -.866,-.5) --  (.866,.5);
\draw[fill] (0,-.25) circle (0.075);
\draw[fill] (-.75,-.75) circle (0.075);
\end{scope}

\begin{scope}[shift ={+(15,0)}]
\draw[dashed] ( .866,.5) -- (.616,-.25) -- (-1.116,-1.25)--( -.866,-.5) ;
\draw[fill, color = white, opacity = .76] ( -.866,-.5) -- (.116,-.75) -- (1.73,.25) -- (.866,.5) -- cycle;
\draw[dashed] ( -.866,-.5) -- (.116,-.75) -- (1.73,.25) -- (.866,.5) -- cycle;\draw[dashed] ( -.866,-.5) -- (-.866,.5) -- ( .866,1.5) -- (.866,.5);
\draw[thick, red] ( -.866,-.5) --  (.866,.5);
\end{scope}

\node  at (2.75,0) {+};
\node  at (7.75,0) {+};
\node  at (13,0) {= $\ a_1$};
\end{scope}

\begin{scope}[shift ={+(5,-6)}]
\draw[dashed] ( .866,.5) -- (.616,-.25) -- (-1.116,-1.25)--( -.866,-.5) ;
\draw[fill, color = white, opacity = .76] ( -.866,-.5) -- (.116,-.75) -- (1.73,.25) -- (.866,.5) -- cycle;
\draw[dashed] ( -.866,-.5) -- (.116,-.75) -- (1.73,.25) -- (.866,.5) -- cycle;\draw[dashed] ( -.866,-.5) -- (-.866,.5) -- ( .866,1.5) -- (.866,.5);
\draw[thick, red] ( -.866,-.5) --  (.866,.5);
\draw[fill] (-.75,-.75) circle (0.075);
\draw[fill] (0,-.25) circle (0.075);
\draw[fill] (-.5,.125) circle (0.075);
\end{scope}

\begin{scope}[shift ={+(10,-6)}]
\draw[dashed] ( .866,.5) -- (.616,-.25) -- (-1.116,-1.25)--( -.866,-.5) ;
\draw[fill, color = white, opacity = .76] ( -.866,-.5) -- (.116,-.75) -- (1.73,.25) -- (.866,.5) -- cycle;
\draw[dashed] ( -.866,-.5) -- (.116,-.75) -- (1.73,.25) -- (.866,.5) -- cycle;\draw[dashed] ( -.866,-.5) -- (-.866,.5) -- ( .866,1.5) -- (.866,.5);
\draw[thick, red] ( -.866,-.5) --  (.866,.5);
\end{scope}

\node  at (8,-6) {= $\ -a_0$};

\node  at (-2.5,0) {(DP1)};
\node  at (-2.5,-3) {(DP2)};
\node  at (-2.5,-6) {(DP3)};
\end{tikzpicture}
\end{figure}
where (DP1), (DP2) and (DP3) are also called \emph{dot permutation relations}.\qed
\end{proposition}

\section{The $\mathfrak{sl}_3$-link homologies via Foams}

In this section we shall review the construction of a link homology theory via webs and foams. This construction consists of three steps. First we need some machinery coming from category theory, namely cubes and abstract complexes. Then, we shall describe how to build a cube from a link diagram, and apply the machinery developed to get a formal ``geometric'' complex of foams and webs. Finally, we make use of Bar-Natan's tautological functors to obtain an honest chain complex and a homology theory.

\subsection{Cubes in categories and abstract complexes}\label{subs:fromcubestcomplexes}

Let us review the construction of the complex in the category of webs and foams associated to an oriented link diagram. To define this complex, a bit of abstract nonsense is necessary.
Denote by $Q_{n}$ the standard $n$-dimensional cube $[0,1]^{n}$. Orient each edge of $Q_n$ from the vertex with the lowest number of $1$s to the vertex with highest number of $1$s.

Let $R$ be a ring. An \emph{$R$-linear category} is a (small) category $\mathbf{C}$ such that, for each pair of objects $A$, $B$, the set of morphisms $Hom_{\mathbf{C}}(A,B)$ has a structure of $R$-module, and the composition is bilinear with respect to this structure. A $\mathbb{Z}$-linear category is often called a \emph{pre-additive category}.

\begin{definition}
A \emph{$n$-cube in a category $\mathbf{C}$}\index{Cube in a category} is the assignment of an object $O_{v}$ to each vertex $v$ of $Q_{n}$ and a morphism $F(v,v^{\prime})\in Hom_{\mathbf{C}}(O_v,O_{v^\prime})$ to the edge from $v$ to $v^\prime$, for each (ordered) pair of vertices $v$, $v^\prime$.
A $n$-cube in a $R$-linear category $\mathbf{C}$ is \emph{commutative} (resp. \emph{skew-commutative}) if for each square the composition of the morphism on the edges commutes (resp. anti-commutes).
 \end{definition}

Given  a commutative cube, it is easy to prove that it is always possible to change the signs of the morphisms on the edges in such a way that the cube becomes skew-commutative (\cite{Khovanov00}).
Now, we wish to assign to a skew-commutative cube a formal chain complex. To do so, one must first give a meaning to a complex over a category. 

\begin{definition}
Let $\mathbf{C}$ be a $R$-linear category. The \emph{category $\mathbf{Com}(\mathbf{C})$ of complexes over $\mathbf{C}$} is the category defined as follows
\begin{enumerate}
\item the objects of $\mathbf{Com}(\mathbf{C})$ are ordered collections of pairs $(C_{i},d_{i})_{i\in\mathbb{Z}}$ where $C_{i} \in Obj(\mathbf{C})$ and $d_{i} \in Hom_{\mathbf{C}}(C_{i},C_{i+1})$ such that
\[ d_{i+1} \circ d_{i} = 0_{Hom_{\mathbf{C}}(C_{i},C_{i+2})};\]
\item the morphisms between two objects $(C_{i},d_{i})$ and $(C^\prime_{j},d^\prime_{j})$ of $\mathbf{Com}(\mathbf{C})$ are collection of maps $(F_{i})_{i\in \mathbb{Z}}$ such that
\[ \forall\: i\in\mathbb{Z}\ :\ F_{i} \in  Hom_{\mathbf{C}}(C_{i},C^\prime_{i+k}) \quad\text{and}\quad F_{i+1} \circ d_{i} = d^\prime_{i+k} \circ F_{i},\]
for a fixed $k\in \mathbb{Z}$, called \emph{degree of $(F_{i})_{i\in\mathbb{Z}}$};
\item the composition of two morphisms $(F_{i})_{i\in \mathbb{Z}}$ and $(G_{j})_{j\in\mathbb{Z}}$ is defined as $(G_{i+k} \circ F_{i})_{i\in\mathbb{Z}}$, where $k$ is the degree of $(F_{i})_{i\in\mathbb{Z}}$.
\end{enumerate}
\end{definition}

In general, an $R$-linear category does not have kernels and co-kernels. So, even though we can define chain complexes we cannot define the homology. However, it is possible to define chain homotopy equivalences.

\begin{definition}
Two morphisms $F$ and $G$ between two objects in $\mathbf{Com}(\mathbf{C})$, say $C_{\bullet} = (C_{i},d_{i})$ and $D_{\bullet} = (D_{i},\partial_{i})$, are (\emph{chain}) \emph{homotopy equivalent} if there exists a morphism $H\in Hom_{\mathbf{Com}(\mathbf{C})}(D_{\bullet},C_{\bullet})$ such that
\[ F - G = d \circ H \pm H \circ \partial. \]
Two objects  $C_{\bullet},\ D_{\bullet}\in Obj(\mathbf{Com}(\mathbf{C}))$ are \emph{homotopy equivalent} if there exists two morphisms
\[ F \in Hom_{\mathbf{Com}(\mathbf{C})}(D_{\bullet},C_{\bullet})\quad \text{and}\quad G\in Hom_{\mathbf{Com}(\mathbf{C})}(C_{\bullet},D_{\bullet})\]
such that the compositions $F \circ G$ and $G \circ F$ are homotopy equivalent to the identity morphism of $D_\bullet$ and $C_\bullet$, respectively.
We will denote by $\mathbf{Com}_{/h}(\mathbf{C})$ the category of complexes over $\mathbf{C}$ and morphisms of $\mathbf{Com}(\mathbf{C})$ up to homotopy equivalence.
\end{definition}

To conclude the abstract construction of a complex from a skew-commutative cube we need one more definition. 

\begin{definition}
Given an $R$-linear category $\mathbf{C}$, the \emph{matrix category over $\mathbf{C}$}\index{Category!- matrix} is the category whose objects are formal direct sums of objects in $\mathbf{C}$ and whose morphisms are matrices with entries in the morphism of $\mathbf{C}$. The composition of two morphisms in the matrix category over $\textbf{C}$ is given by the usual matrix multiplication rule.
\end{definition}

Denote by $\mathbf{Kom(C)}$ the category of complexes over the matrix category over $\mathbf{C}$, where $\mathbf{C}$ is an arbitrary $R$-linear category. Given a skew-commutative $n$-cube $Q$ in $\mathbf{C}$, define
\[ C_{i}^{Q} = \bigoplus_{\vert v \vert = i} Q_{v}\quad d^{Q}_{i} = \sum_{\vert v \vert = i} \bigoplus_{v^\prime} F (v,v^\prime),\]
where $F(v,v^\prime)$ is defined to be zero if there is no edge from $v$ to $v^\prime$, and $\vert v \vert$ denotes the number of $1$'s in $v$. It is an easy verification that $(C_{i}^{Q},d_{i}^{Q})$ is an object in $\mathbf{Kom(C)}$.

\subsection{The Khovanov-Kuperberg bracket and the geometric complex}\label{subsection:homologies}

Fix a potential $\omega(x) \in R[x]$. 
Let $D$ be an oriented link diagram. Fix an order of the crossings of $D$, say $\{ c_{1} ,...,c_{k} \}$. Each crossing has two possible \emph{web resolutions}, see Figure \ref{fig:Webresolutions}.
These resolutions come with an integer depending on the crossing and the type of resolution performed.

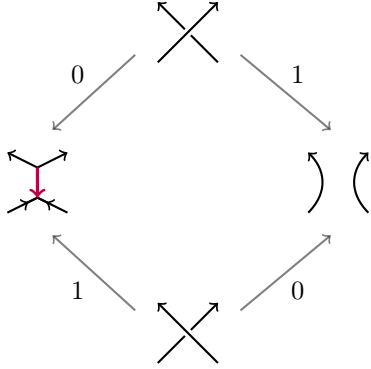
\begin{figure}[]
\centering
\begin{tikzpicture}[scale = 2]
\draw[->, thick] (2.2,-.2) .. controls +(-.125,.125) and +(-.125,-.125) .. (2.2,.2);
\draw[->, thick] (1.8,-.2) .. controls +(.125,.125) and +(.125,-.125) .. (1.8,.2);

\draw[->, thick] (0,.1) -- (-.2,.2);
\draw[->, thick] (0,.1) --(.2,.2);
\draw[thick] (0,-.1) -- (-.2,-.2);
\draw[thick, -<] (0,-.1) --(-.1,-.15);
\draw[thick]  (0,-.1) --(.2,-.2);
\draw[thick, -<] (0,-.1) --(.1,-.15);
\draw[very thick, purple, <-] (0,-.1) --(0, .1);

\draw[<-, thick] (.8, 1.2) --(1.2,.8);
\pgfsetlinewidth{8*\pgflinewidth}
\draw[white] (.8, .8) --(1.2,1.2);
\pgfsetlinewidth{.125*\pgflinewidth}
\draw[->, thick] (.8, .8) --(1.2,1.2);

\draw[<-, thick] (1.2, -.8) --(.8,-1.2);
\pgfsetlinewidth{8*\pgflinewidth}
\draw[white] (1.2, -1.2) --(.8,-.8);
\pgfsetlinewidth{.125*\pgflinewidth}
\draw[->, thick] (1.2, -1.2) --(.8,-.8);

\draw[opacity = .5,->,thick] (.65, .85) -- (.1,.35);
\draw[opacity = .5,->,thick] (1.35, .85) -- (1.95,.35);
\draw[opacity = .5,->,thick] (.65, -.85) -- (.1,-.35);
\draw[opacity = .5,->,thick] (1.35, -.85) -- (1.95,-.35);

\node[above left ] at (.375,.6) {0};
\node[below left ] at (.375,-.6) {1};
\node[above right] at (1.625,.6) {1};
\node[below right] at (1.625,-.6) {0};
\end{tikzpicture}
\caption{Web resolutions of positive (top) and negative (bottom) crossings.}
\label{fig:Webresolutions}
\end{figure}

There is a natural bijection between the vertices of a $k$-dimensional cube $[0,1]^{k}$ and the web resolutions of the diagram\footnote{That is the web obtained by replacing each crossing with a web resolutions.} $D$; to $v\in \{ 0,1\}^{n}$ is associated the web $W(D,v)$ obtained by replacing, for each $i$, the crossing $c_{i}$ with its $v_{i}$-web resolution.

To each oriented edge $v \to v^\prime$ of $Q_{k}$ is associated a foam $F(v, v^\prime)$ between the webs $W(D,v)$ and $W(D,v^\prime)$. The foam $F(v,v^\prime)$ is everywhere a cylinder, except above a disk where the two webs differ, where the cobordism looks like one of the two elementary web cobordisms depicted Figure \ref{fig:Elementarwebcob}. 
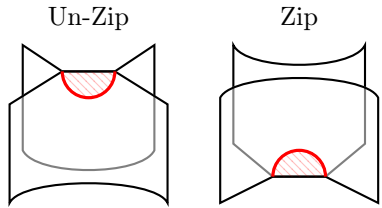
\begin{figure}[h]
\centering
\begin{tikzpicture}[scale=.35, thick]

\draw (.5,.25) -- (2,-1) -- (4,-1) -- (5.5,.25) -- (5.5,4) .. controls +(-.5,-1) and +(.5,-1) ..(.5,4) --cycle;
\draw[fill, white, opacity = .5] (0,-2) -- (2,-1) -- (4,-1) -- (6,-2) -- (6,2) .. controls +(-.5,1) and +(.5,1) ..(0,2) --cycle;
\draw (0,-2) -- (2,-1) -- (4,-1) -- (6,-2) -- (6,2) .. controls +(-.5,1) and +(.5,1) ..(0,2) --cycle;
\draw[pattern color = red, pattern= north west lines, opacity = .5] (4,-1) -- (2,-1) arc (180:0:1);
\draw[very thick,red]  (2,-1) arc (180:0:1);

\begin{scope}[shift = {+(-2,2)}, rotate = 180]
\draw (.5,-2) -- (2,-1) -- (4,-1) -- (5.5,-2) -- (5.5,2) .. controls +(-.5,1) and +(.5,1) ..(.5,2) --cycle;
\draw[fill, white, opacity = .5] (0,.25) -- (2,-1) -- (4,-1) -- (6,.25) -- (6,4) .. controls +(-.5,-1) and +(.5,-1) ..(0,4) --cycle;
\draw (0,.25) -- (2,-1) -- (4,-1) -- (6,.25) -- (6,4) .. controls +(-.5,-1) and +(.5,-1) ..(0,4) --cycle;

\draw[pattern color = red, pattern= north west lines, opacity = .5] (4,-1) -- (2,-1) arc (180:0:1);
\draw[very thick,red]  (2,-1) arc (180:0:1);

\node at (3,-3) {Un-Zip};
\end{scope}
\node at (3,5) {Zip};
\end{tikzpicture}
\caption{Elementary web cobordisms. The arc in red is a singular arc.}
\label{fig:Elementarwebcob}
\end{figure}
The cube we defined depends on the choice of an ordering of the crossings of $D$. Moreover, it is easy to see that, by a Morse-theoretic argument, the cube associated to an oriented link diagram is commutative. We can turn this commutative cube into a skew commutative cube $Q(D)$ (but there is the choice of signs for the edges). 
Finally, using the abstract construction described in Subsection \ref{subs:fromcubestcomplexes} we can associate a complex $\langle D \rangle_\omega$ in $\mathbf{Kom}(\mathbf{Foam}_{/\ell})$ to the skew commutative cube $Q(D)$.

\begin{rem*}
The complex $\langle D \rangle_{\omega}$ does not depend on the potential \emph{per se}; what really depends on $\omega$ is the category $\mathbf{Foam}_{/\ell}$.
\end{rem*}

The complex $\langle D \rangle_{\omega}$ is called the \emph{Khovanov-Kuperberg bracket of $D$} (with respect to $\omega$). Whenever $\omega$ is fixed or clear from the context we will remove it from the notation. With some standard machinery of homological algebra it is easy to show the following proposition.

\begin{proposition}[\cite{Khovanov03, Mackaayvaz07}]
The Khovanov-Kuperberg bracket of an oriented link diagram $D$ does not depend (up to isomorphism in $\mathbf{Kom}(\mathbf{Foam}_{/\ell})$) on the sign assignment or the order of the crossings used to obtain the cube $Q(D)$.\qed
\end{proposition}

The Khovanov-Kuperberg bracket is the $\mathfrak{sl}_3$-analogue of the Khovanov bracket introduced by Bar-Natan in \cite{BarNatan05cob}. Exactly as in the case of the Khovanov bracket, to turn the Khovanov-Kuperberg bracket into an invariant of links (and not framed links) we need to shift the homological grading.

\begin{definition}
Given an $R$-linear category $\mathbf{C}$ and $C_\bullet = (C_{i},d_{i})_{i\in\mathbb{Z}}\in \mathbf{Com}(\mathbf{C})$, the \emph{shift of $C_{\bullet}$ by $k\in \mathbb{Z}$} is the object $C_{\bullet}(k)$ in $ \mathbf{Com}(\mathbf{C})$ defined as follows:
\[ C_\bullet(k) = (C_{i+k},d_{i+k})_{i\in\mathbb{Z}}.\]
\end{definition}

Finally, we can define the \emph{the geometric $\mathfrak{sl}_{3}$-complex of $D$} (with respect to $\omega$) as follows
\[\widetilde{C}_{\omega}^{\bullet}(D,R) = \langle D\rangle_{\omega}(-n_{+}(D)),\]
where $n_+(D)$ (resp. $n_{-}(D)$) is the number of  positive (resp. negative) crossings in $D$ (cf. Figure \ref{fig:Webresolutions}).
This is a link invariant in the sense of the following proposition.

\begin{theorem}(Mackaay-Vaz, \cite{Mackaayvaz07})\label{theorem:invarianceofthegeometriccomplex}
Let $\omega$ be a potential. If $D$ and $D^\prime$ are oriented link diagram representing the same link, then $\widetilde{C}_{\omega}^{\bullet}(D,R)$ and $\widetilde{C}_{\omega}^{\bullet}(D^\prime,R)$ are chain homotopy equivalent.
Furthermore, the assignment
\[ \widetilde{C}_{\omega}^{\bullet}: \mathbf{Link} \longrightarrow \mathbf{Kom}_{/\pm h}(\mathbf{Foam}_{/\ell}),\] 
is a functor from the category $\mathbf{Link}$ (i.e. the category of links in $\mathbb{R}^3$, and properly embedded surfaces in $\mathbb{R}^3 \times [0,1]$ up to boundary-fixing isotopies), to the category $\mathbf{Kom}_{/\pm h}(\mathbf{Foam}_{/\ell})$ (i.e. the category obtained from $\mathbf{Kom}_{h}(\mathbf{Foam}_{/\ell})$ by considering the morphisms up to sign).
\qed
\end{theorem}

\subsection{Tautological functors and the $\mathfrak{sl}_3$-homology}
The category $\mathbf{Mat}(\mathbf{Foam}_{/\ell})$ is not an Abelian category. Thus, it is not possible to define the homology of $\widetilde{C}_{\omega}^{\bullet}(D,R)$. There are different ways to turn the geometric complex into an honest chain complex. Out of the different possibilities, following \cite{Mackaayvaz07}, we pursue the approach via tautological functors.
 
\begin{definition}
The \emph{tautological functor} is the functor
\[T: \mathbf{Foam}_{/\ell} \longrightarrow R-\mathbf{Mod}\]
defined on an object $W^\prime\in Obj(\mathbf{Foam}_{/\ell})$ by
\[ T(W^\prime) = Hom_{\mathbf{Foam}_{/\ell}}(\emptyset,W^\prime)\]
and on morphisms by composition on the left, that is
\begin{align*}
T(F):\ T(&W^\prime) \longrightarrow\quad T(W^{\prime\prime})\\
& G\quad \longmapsto\quad F \circ G
\end{align*}
for each $F\in Hom_{\mathbf{Foam}_{/\ell}}(W^\prime,W^{\prime\prime})$ and $W^\prime,\: W^{\prime\prime}\in Obj(\mathbf{Foam}_{/\ell})$.
\end{definition} 

Note that, if we have a disjoint union of the webs $W^\prime$ and $W^{\prime\prime}$, then
\[T (W^\prime \sqcup W^{\prime\prime}) \simeq T (W^\prime) \otimes_{R} T(W^{\prime\prime})\]
as $R$-modules. Before proceeding further let us show in an example how the functor $T$ works. This example will be useful later on. 

\begin{example}\label{ex:Tcircle}
Let us compute $T(\bigcirc)$ (i.e.\: find its isomorphism class as an $R$-module).

By definition $ T(W^\prime) = Hom_{\mathbf{Foam}_{/\ell}}(\emptyset,W^{\prime})$ is the $R$-module generated by all foams bounding $W^\prime$ (modulo local relations).
All closed components of such a foam evaluate to elements of $R$. Thus, $T(\bigcirc)$ is generated (as $R$-module) by connected foams. 
Since these foams must bound the circle $\bigcirc$, which is made of regular boundary points, we can use the genus reduction relation and write any connected foam bounding  $\bigcirc$ as an $R$-linear combinations of disks marked with at most two dots. It follows that $T(\bigcirc )$ is generated by dotted disks.
Now, consider the epimorphism of $R$-modules
\[ \Phi : R[x] \longrightarrow Hom_{\mathbf{Foam}_{/\ell}}\left(\emptyset,\bigcirc\right)\]
mapping $x^k$ to the disk with $k$-dots. 
The dot reduction relation tells us that $\omega (x)$ is in the kernel. On the other hand, the disk with two dots, the disk with a single dot and the disk with no dots are linearly independent over $R$ (use the pairing to get a linear system). Finally, an easy application of the Euclidean division algorithm (which works in any ring, assuming that the polynomial by which we are dividing has invertible leading term, see \cite[Theorem 1.1 Section IV]{Lang05}) shows that $Ker(\Phi)$ is exactly $(\omega(x))$, and thus
\[ T(\bigcirc) \simeq \frac{R[x]}{(\omega (x))}\]
\end{example}

There is a natural way to extend the tautological functor to the category $\mathbf{Kom}(\mathbf{Foam}_{/\ell})$ (cf. \cite[Section 9]{BarNatan05cob}). With an abuse of notation we shall denote the extended functor also by $T$.

\begin{definition}
The \emph{$\mathfrak{sl}_{3}$-complex} (with respect to $\omega$) of an oriented link diagram $D$ is
\[ C_{\omega}^{\bullet}(D,R) = T\left(\widetilde{C}_{\omega}^{\bullet}(D,R) \right)\in Obj(\mathbf{Kom}(R-\mathbf{Mod})).\]
\end{definition}

The following proposition is an immediate consequence of Theorem \ref{theorem:invarianceofthegeometriccomplex}.

\begin{proposition}
The isomorphism class (as $R$-module) of $H_{\omega}^\bullet(D,R)$ is a link invariant, so we may write $H_{\omega}^\bullet(L,R)$ where $L$ is the oriented link represented by $D$.
Moreover, $H_{\omega}^{\bullet}$ defines a functor between the category $\mathbf{Link}$ and the category $R$-$\mathbf{Mod}_{gr}$ of graded $R$-modules.
\qed
\end{proposition}

The homology of the $\mathfrak{sl}_3$-complex will be called \emph{$\mathfrak{sl}_{3}$-homology of $L$} (with respect to $\omega$).

\subsection{Graded and filtered $\mathfrak{sl}_3$-homologies}\label{sec:grading}

To conclude the background material, we wish to describe how to define a second grading or a filtration on the $\mathfrak{sl}_3$-complex. Suppose that $R$ is a graded ring, for the sake of simplicity we shall assume $R$ to be graded over the non-negative integers. By setting $deg(x) = 2$, the grading on $R$ induces a grading on $R[x]$. The choice of an homogenous potential $\omega(x)$ (i.e. $deg(a_{i}) = 2(3 - i)$) induces a graded structure on the category $\mathbf{Foam}_{/\ell}$ (see \cite[Definition 6.1]{BarNatan05cob} and \cite[Section 2]{Mackaayvaz07}); that is the modules $ Hom_{\mathbf{Foam}_{/\ell}}(W^\prime, W)$ (and thus also $T(W)$) become graded $R$-modules. This structure is defined by setting
\[deg_{\mathbf{Foam}}(F) = -2\chi(F)  + \chi(\partial F) + 2 d,\]
for each foam $F$ with $d$ dots.

In particular, it follows that the $\mathfrak{sl}_{3}$-complex associated to an oriented diagram $D$ and a potential $\omega$ is graded. Furthermore, the differential respects the \emph{quantum degree} $qdeg$, which is defined as follows
\[ qdeg(x) = deg_{\mathbf{Foam}}(x) - \vert v \vert + 3n_{+}(D) - 2n_{-}(D), \]
where $x$ is an homogeneous element of $T(W(D,v))$, and $v$ a vertex of the cube $Q_{n_{+}+n_{-}}$.

Now, suppose that $R$ is trivially graded (i.e. supported in degree $0$). In this case, the unique homogeneous potential is $x^{3}$ (which corresponds to the original theory due to Khovanov). For all the other potentials the quantum grading is not well defined (because the reduction relations are not homogenous). Nonetheless, one may define a filtered structure on $ Hom_{\mathbf{Foam}_{/\ell}}(W^\prime, W)$ by setting 
\[\mathscr{F}_{i} Hom_{\mathbf{Foam}_{/\ell}}(W^\prime, W) = \langle\: F,\ deg_{\mathbf{Foam}}(F) \leq i \rangle_{R},\]
where $deg_{\mathbf{Foam}}$ is defined as above.
This filtered structure extends to $\mathbf{Kom}(\mathbf{Foam}_{/\ell})$, and also to the $\mathfrak{sl}_3$-complex, and the differential does not increase the filtration level (shifted as in the case of the quantum degree). This (shifted) filtration on the $\mathfrak{sl}_3$-complex is called \emph{quantum filtration}.

\begin{rem*}\label{rem:gradingcircle}%
One can prove that the isomorphism in Example \ref{ex:Tcircle} induces the following isomorphism of graded (resp. filtered) $R$-modules
\[ T\left( \bigcirc \right) \simeq \frac{R[x]}{(\omega(x))}(-2),\]
where $deg(x) = 2$, and in the filtered case the filtration on the left-hand side is induced by the degree.%
\end{rem*}

To conclude this section we state the following result, due to Mackaay and Vaz (\cite[Lemma 2.9]{Mackaayvaz07}, compare also with \cite{Khovanov03})

\begin{proposition}(Khovanov-Kuperberg relations)\label{proposition:KhovKuprelgraded}
We have the following isomorphisms of graded (filtered) $R$-modules.
\begin{itemize}
\item[] $\text{(circle removal)}\qquad\ T(W^\prime \sqcup \bigcirc) \simeq T(\bigcirc) \otimes T(W^\prime)$
\item[] $\text{(digon removal)}\qquad\ T(W_1 ) \simeq T(W_2)(-1) \oplus T(W_2)(1)$
\item[] $\text{(square removal)}\qquad T(W^\prime_1) \simeq T(W^\prime_2) \oplus T(W^\prime_3)$
\end{itemize}
where $W_1$ and $W_2$ (resp. $W^\prime_1$, $W^\prime_2$ and $W^\prime_3$) are two (resp. three) webs which are identical but in a small ball where they are as depicted in Figure \ref{fig:KhovanovKuperbergrel}, and $(\cdot)$ indicates the degree (filtration) shift.\qed
\end{proposition}
\begin{figure}[H]
\centering
\begin{tikzpicture}[scale = .5]
\begin{scope}[shift = {+(1.75,2)}]
\draw[thick,->] (1.5,0) -- (2,0);
\draw[thick] (2,0) .. controls +(.25,.25) and +(-.25,0) .. (3,.45);
\draw[thick,<-] (3,.45) .. controls +(.25,0) and +(-.25,.25) .. (4,0);
\draw[thick] (2,0) .. controls +(.25,-.25) and +(-.25,0) .. (3,-.45);
\draw[thick,<-] (3,-.45) .. controls +(.25,0) and +(-.25,-.25) .. (4,0);
\draw[thick,->] (4,0) -- (4.5,0);
\end{scope}

\draw[thick,->] (7.5,2) -- (10.5,2);
\draw[dashed] (9,2) circle (1.5);
\node at (9,0) {$W_{2}$};
\draw[dashed] (4.75,2) circle (1.5);
\node at (4.75,0) {$W_{1}$};

\draw[thick,->] (2,-1) -- (2.5,-1.5);
\draw[thick,->] (3.5,-1.5) -- (4,-1);
\draw[thick,<-] (2,-3) -- (2.5,-2.5);
\draw[thick,<-] (3.5,-2.5) -- (4,-3);
\draw[thick,->] (3.5,-1.5)-- (3,-1.5);
\draw[thick] (3,-1.5) -- (2.5,-1.5);
\draw[thick] (3.5,-2) -- (3.5,-2.5);
\draw[thick,->] (3.5,-1.5) -- (3.5,-2);
\draw[thick,<-] (2.5,-2) -- (2.5,-2.5);
\draw[thick] (2.5,-1.5) -- (2.5,-2);
\draw[thick,<-] (3,-2.5) -- (2.5,-2.5);
\draw[thick] (3.5,-2.5) -- (3,-2.5);
\draw[dashed] (3,-2) circle (1.5);

\draw[thick,<-] (5.75,-3) .. controls +(.5,.5) and +(.5,-.5) .. (5.75,-1);
\draw[thick,->] (7.75,-3) .. controls +(-.5,.5) and +(-.5,-.5) .. (7.75,-1);
\draw[dashed] (6.75,-2) circle (1.5);

\draw[thick,->] (9.5,-1) .. controls +(.5,-.5) and +(-.5,-.5) .. (11.5,-1);
\draw[thick,<-] (9.5,-3) .. controls +(.5,.5) and +(-.5,.5) .. (11.5,-3);
\draw[dashed] (10.5,-2) circle (1.5);
\node at (3,-4.25) {$W^\prime_{1}$};
\node at (6.75,-4.25) {$W^\prime_{2}$};
\node at (10.5,-4.25) {$W^\prime_{3}$};
\end{tikzpicture}
\caption{Webs involved in the Khovanov-Kuperberg relations.}
\label{fig:KhovanovKuperbergrel}
\end{figure}
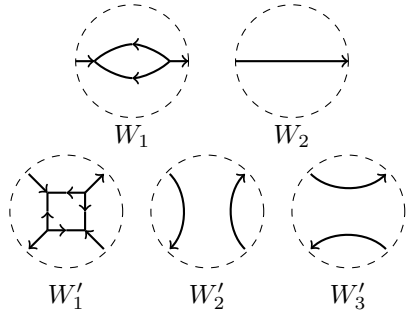

\section{Transverse invariants and the universal $\mathfrak{sl}_3$-theory}

Let $R$ be an integral domain, and fix a potential $\omega (x)\in R[x]$. In this section we define a family of transverse braid invariants in $C_{\omega}^{\bullet}(\overline{B},R)$, where $B$ is a closed braid diagram and the overline denotes the mirror. The elements of this family are in bijection with the (distinct) roots of $\omega$ in $R$. From now on, unless otherwise stated, all tensor products are assumed to be taken over $R$ and all the isomorphisms are assumed to be isomorphisms of $R$-modules.
\subsection{The $\beta$-chains}
Assume that $\omega(x)$ has a root $x_1$ in $R$. It follows that
\[\omega(x) = x^3 + a_2 x^2 + a_1 x + a_0 = (x- x_1)(x^2 + a_1^\prime x + a_0^\prime),\]
for some $a_1^\prime$, $a_{0}^\prime \in R$ such that
\begin{equation}
 a_2 = a_1^\prime - x_1 \qquad a_1 = a_0^\prime - x_1 a_1^\prime \qquad a_{0} = -x_1 a_{0}^\prime .
\label{eq:relationcoeff}
\end{equation}
Let $D$ be an oriented link diagram. Define the \emph{oriented web resolution} $\underline{w}_{D}$ to be the web resolution where each positive crossing is replaced by its $1$-web resolution, and every negative crossing is replaced by its $0$-resolution. In other words, the oriented web resolution is the web resolution where both \ocrosstextp and \ocrosstextn are replaced by \orisplittext . It follows that the oriented web resolution is a collection of loops. Moreover, these loops have a natural orientation induced by the orientation of $D$. 

\begin{definition}
Let $D$ be a oriented link diagram.
Consider a family of disjoint unknotted disks $\{ \mathbb{D}_{\gamma}\}_{\gamma\in \underline{w}_{D}}$ properly embedded in $(\mathbb{R}^{2}\times \{0\})\times [0,1] \subseteq \mathbb{R}^3 \times [0,1]$, obtained by pushing the Jordan disks bounding $\underline{w}_{D}$ in $\mathbb{R}^2\times \{ 0 \}$. Denote by $\mathbb{D}^{k}_{\gamma}$ the disk $\mathbb{D}_{\gamma}$ with $k$ dots on it.
The \emph{$\beta$-chain (with respect to $\omega$) associated to the root $x_{1}$} is defined as follows
\[ \beta_{\omega,x_1}(D)= \sum_{S \subseteq \underline{w}_{D}} \sum_{S^\prime \subseteq \underline{w}_{D}\setminus S}  (a_{1}^{\prime})^{\# S^\prime }(a_{0}^{\prime})^{\# S }\left(\bigsqcup_{\gamma\in \underline{w}_{D}\setminus( S \cup S^\prime )} \mathbb{D}_{\gamma}^{2} \sqcup \bigsqcup_{\gamma\in S^\prime} \mathbb{D}_{\gamma}^{1} \sqcup \bigsqcup_{\gamma\in S} \mathbb{D}_{\gamma}^{0}\right)\in T(\underline{w}_{D}),\]
where $\# S$ denotes the number of elements in $S$.
\end{definition}

By definition of the $\mathfrak{sl}_3$-complex, to each web resolution $\underline{w}$ of $D$ corresponds a direct summand $T(\underline{w})$ in $C_{\omega}^{\bullet}(D,R)$ in homological degree $ - n_{+}(D) + \vert \underline{w} \vert$ where $n_{+}(D)$ is the number of positive crossings in $D$, and $\vert \underline{w} \vert $ is the number of $1$-web resolutions in the web resolution $\underline{w}$. In particular, we have
\[ T(\underline{w}_{D}) \subseteq C_{\omega}^{0}(D,R).\]
Furthermore, the filtered quantum degree of $\beta_{\omega,x_{1}}$ can be easily computed. It suffices\footnote{We are using the fact that the disks with dots are a filtered basis for $T(\bigcirc \cdots \bigcirc)$. This follows immediately from the observation that the isomorphism (cf. Equation \eqref{eq:algdefofb3})
\[T(\underbrace{\bigcirc \cdots \bigcirc}_{k}) \simeq \frac{R[x_{1},...,x_{k}]}{(\omega(x_{1}), \dots , \omega(x_{k}))}(-2k),\]
given by sending the disk with $d$ dots bounding the $i$-th circle into $x_{i}^{d}$, is an isomorphism of filtered modules. Here the filtration on $R[x_{1},...,x_{k}]/(\omega(x_{1}), \dots , \omega(x_{k}))$ is the one induced by the total grading, $deg(x_{i})=2$ for each $i$, and $(-2k)$ denotes a shift of $-2k$ in the filtered degree (cf. Remark \ref{rem:gradingcircle})}. to look at the maximal quantum degree of the summands in the definition of $\beta_{\omega,x_{1}}$. The maximal degree is achieved by the summand where the disks have two dots each. Thus, the filtered quantum degree of $\beta_{\omega,x_{1}}(D)$ is $2 \# \text{disks in } \underline{w}_{D} + 2 w(D)$. In the special case $D = \widehat{B}$, we have that 
\[2 \# \text{disks in } \underline{w}_{\widehat{B}} + 2 w(B)  = 2( b(B)- w (\overline{B})) = -2sl(\overline{B}),\]
where the last equality is due to the fact that $b(B)= b(\overline{B})$.
The same computation works in the graded case.

From the Khovanov-Kuperberg relations and from Example \ref{ex:Tcircle}, it follows that, as $R$-modules,
\begin{equation}
 T(\underline{w}_{D}) \simeq \bigotimes_{\gamma \in \underline{w}_{D}} \frac{R[x_{\gamma}]}{(\omega (x_\gamma))} \simeq \frac{ R[x_\gamma\:\vert\:\gamma\in\underline{w}_{D}]}{\left(\omega (x_{\gamma})\right)_{\gamma\in\underline{w}_{D}}}\label{eq:algdefofb3}
\end{equation}
where $\gamma\in \underline{w}_{D}$ should be read as ''$\gamma$ is a circle in $\underline{w}_{D}$''. It is easy to see that the isomorphism in \eqref{eq:algdefofb3} maps $\beta_{\omega,x_1}(D)$ to
\[ \prod_{\gamma\in \underline{w}_{D}} [x_{\gamma}^{2} + a_{1}^\prime x_{\gamma} + a_{0}^{\prime}]\in T(\underline{w}_{D}) \subseteq C_{\omega}^{0}(D,R).\]
In the rest of the paper we shall freely switch between these two representations of the chains $\beta_{\omega,x_1}(D)$.

\begin{rem*}
The multiplication of $\beta_{\omega,x_1}(D)$ by $x_{\gamma}$, which is an algebraic operation, corresponds ``geometrically`` to adding a dot to the disk $\mathbb{D}_{\gamma}$ in each summand of the ``geometric expression'' of $\beta_{\omega,x_1}(D)$.
\end{rem*}

\begin{figure}[H]
\centering
\begin{tikzpicture}[scale =.25, thick]
\draw[dashed] (6,0) arc (0:180:6 and 1);
\draw (6,0) arc (0:-180:6 and 1);

\draw (-4,8) circle (2 and .333);
\draw (4,8) circle (2 and .333);

\draw[pattern = north west lines, pattern color =red, dashed, opacity =.8] (-1.55291427061512,0.965925826289068) .. controls +(0,3) and +(-.5,0) .. (0,5) .. controls +(.75,0) and +(0,4.5) .. (1.55291427061512,-0.965925826289068) --cycle;

\draw[white] (-1.55291427061512,0.965925826289068) .. controls +(0,3) and +(-.5,0) .. (0,5) .. controls +(.75,0) and +(0,4.5) .. (1.55291427061512,-0.965925826289068) --cycle;
\draw[thick, dashed ,red] (1.55291427061512,-0.965925826289068) -- (-1.55291427061512,0.965925826289068) .. controls +(0,3) and +(-.5,0) .. (0,5);
 \draw[thick,red] (0,5) .. controls +(.75,0) and +(0,4.5) .. (1.55291427061512,-0.965925826289068) ;
\draw (2,8) .. controls +(0,-1) and +(1,0) .. (0,5) .. controls +(-1,0) and +(0,-1) .. (-2,8);
\draw (-6,8) -- (-6,0);
\draw (6,8) -- (6,0);
\end{tikzpicture}
\caption{The foam $F$.}\label{fig:foamlemma}
\end{figure}
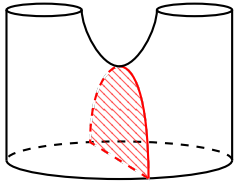
\begin{lemma}\label{lemma:dbetavanish}
Let $F$ be the foam in Figure \ref{fig:foamlemma}, and consider the morphism of $R$-modules
\[T(F): T(\bigcirc \sqcup \bigcirc ) \longrightarrow T(W_\theta),\]
where $W_{\theta}$ is the theta web (i.e. the closure of the web $W_{1}$ in Figure \ref{fig:KhovanovKuperbergrel}).
Define $\beta $ as follows
\[ \beta = (x^{2} + a_{1}^\prime x + a_0^\prime) (y^{2} + a_{1}^\prime y + a_0^\prime) \in \frac{R[x,y]}{(\omega(y),\omega(x))} \simeq T(\bigcirc \sqcup \bigcirc),\]
we have
\[T(F)(\beta) = 0. \]
\end{lemma}
\begin{proof}
Our aim is to prove that $T(F)(\beta)$ can be written as follows
\[ T(F)(\beta) = \sum_{j=1}^{k} c_j(F_{j}^{\prime\prime\prime} + a_2 F_{j}^{\prime\prime} + a_1 F_{j}^\prime + a_{0}F_{j}),\quad c_j \in R \]
where $F_{j},\: F_{j}^\prime,\: F_{j}^{\prime\prime},\: F_{j}^{\prime\prime\prime}$ are of the form shown in Figure \ref{fig:prefoam}, and are identical except in a small region where they differ as shown in Figure \ref{fig:F_j}.
\begin{figure}[H]
\centering
\begin{tikzpicture}[scale = .25]

\draw[dashed] (0,7.5) circle (2);
\draw[dashed] (7,7.5) circle (2);
\draw[dashed] (14,7.5) circle (2);
\draw[dashed] (21,7.5) circle (2);

\node at (21,3.5)  {$F_{j}$};
\node at (14,3.5)  {$F_{j}^{\prime}$};
\node at (7,3.5)  {$F_{j}^{\prime\prime}$};
\node at (0,3.5)  {$F_j^{\prime\prime\prime}$};

\draw[fill] (7.5,7.5) circle (0.15) ;
\draw[fill] (6.5,7.5) circle (0.15) ;
\draw[fill] (0,7.5) circle (0.15) ;
\draw[fill] (0.75,7.5) circle (0.15) ;
\draw[fill] (-.75,7.5) circle (0.15) ;
\draw[fill] (14,7.5) circle (0.15) ;
\end{tikzpicture}
\caption{The local difference between the foams $F_{j},\: F_{j}^\prime,\: F_{j}^{\prime\prime},\: F_{j}^{\prime\prime\prime}$.}
\label{fig:F_j}
\end{figure}
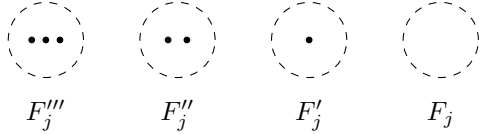
Since, for each $j$, $(F_{j}^{\prime\prime\prime} + a_2 F_{j}^{\prime\prime} + a_1 F_{j}^\prime + a_{0}F_{j})$ is trivial in $\mathbf{Foam}_{/\ell}$ by (DR), the claim shall follow.
To avoid graphical calculus we make use of polynomials. So, let us denote by the monomial $A^{r}B^{s}C^{t}$ the foam (in $(\mathbb{R}^{2}\times \{0\} )\times [0,1]$) shown in Figure \ref{fig:prefoam}, where $s$, $r$ and $t$ indicate the number of dots in the regions $A$, $B$ and $C$ respectively.
By definition $T(F)(\beta)$ can be written as follows
\begin{small}
\[T(F)(\beta) = A^{2}B^{2} + a_{1}^{\prime}(A^2B + AB^2) + ( a_{1}^\prime +1)^2 AB + a_{0}^\prime(A^{2} + B^{2}) + a_{1}^\prime a_{0}^\prime (A + B)+(a_{0}^\prime)^{2}.\]
\end{small}
With this notation we can write the dot permutation relations (DP1), (DP2) and (DP3) described in Proposition \ref{proposition:Mackaayvazrel} as follows:
\begin{equation}
\tag{DP1}
 A + B + C = - a_2
 \label{eq:lineareq}
\end{equation}
\begin{equation}
\tag{DP2}
 AC + BC + AB  =  a_1
 \label{eq:quadraticeq}
\end{equation}
\begin{equation}
\tag{DP3}
 ABC = -a_0
 \label{eq:cubiceq}
\end{equation}
Since all foam relations are local, and since we are allowed to move the dots inside regions, the formal products above satisfy associativity. Using Relations \eqref{eq:lineareq}, \eqref{eq:quadraticeq} and \eqref{eq:cubiceq} we obtain
\begin{small}
\[ T(F)(\beta) = B(A^3 + a_2A^2  +a_1 A + a_0) (A - a_1^\prime + 1) = \]
\[ = \left(BA^4 + a_2 BA^3  +a_1 BA^2 + a_0BA\right) + (- a_1^\prime + 1)(A^3 + a_2A^2  +a_1 A + a_0), \]
\end{small}
which is the desired decomposition of $T(F)(\beta)$.
\end{proof}
\begin{figure}[h]
\centering
\begin{tikzpicture}[scale=.65, thick]
\draw[dashed] (-6,0) arc (180:55:3) arc (55:6:3);
\draw (-6,0) arc (180:55:3);
\draw (-6,0) arc (-180:-82:3 and .5);
\draw[dashed] (-6,0) arc (180:22:3 and .5);
\draw (3,-.25) arc (0:128:3);
\draw (3,-.25) arc (0:-150:3 and .5);
\draw[fill, white, opacity =.5, dashed] (0,.2) .. controls +(-.45,1.5)  and +(.3,0).. (-1.85,2.11) .. controls +(-.3,0) and +(-.45,1.5) .. (-2.6,-.5) -- cycle;
\draw[pattern color=red, pattern = north east lines, opacity =.5, dashed] (0,.2) .. controls +(-.45,1.5)  and +(.3,0).. (-1.85,2.11) .. controls +(-.3,0) and +(-.45,1.5) .. (-2.6,-.5) -- cycle;
\draw[white, thick] (0,.2) .. controls +(-.45,1.5)  and +(.3,0).. (-1.85,2.11) .. controls +(-.3,0) and +(-.45,1.5) .. (-2.6,-.5);
\draw[red, thick, dashed] (0,.2) .. controls +(-.45,1.5)  and +(.3,0).. (-1.85,2.11);
\draw[red, thick] (-1.85,2.11) .. controls +(-.3,0) and +(-.45,1.5) .. (-2.6,-.5);
\draw[dashed] (3,-.25) arc (0:91:3 and .5);
\node at (0,2) {B};
\node at (-3.25,2.25) {A};
\node at (-1.5,1) {C};
\draw[fill, white] (-4,.7) rectangle (-3.4,-.3);
\node at (-3.75,.75) {$\underbrace{\bullet\ ...\ \bullet}_{r}$};
\draw[fill, white] (1.8,.7) rectangle (1.2,-.3);
\node at (1.5,.75) {$\underbrace{\bullet\ ...\ \bullet}_{s}$};
\draw[fill, white] (-1.8,-.15) rectangle (-1.2,-.5);
\node at (-1.5,0) {$\underbrace{\textcolor{red}{\bullet\ ...\ \bullet}}_{t}$};
\end{tikzpicture}
\caption{The foam $A^{r}B^{s}C^{t}$.}\label{fig:prefoam}
\end{figure}
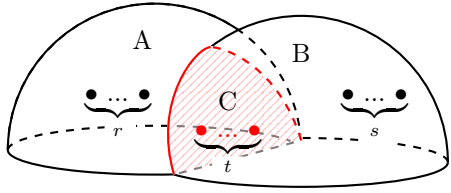

%\textcolor{red}{
\begin{rem*}
In the proof of Lemma \ref{lemma:dbetavanish} we identified $T(W_{\theta})$ with the $R$-module
\[ M = \frac{R[A,B,C]}{\big(\text{\eqref{eq:lineareq}, \eqref{eq:quadraticeq}, \eqref{eq:cubiceq}}\big)}.\]
\begin{enumerate}
\item The above identification of $T(W_\theta)$ with $M$ completely disregards the quantum grading (resp. filtration). Taking into account the quantum degree (resp. filtration) there is a shift one has to consider. More precisely, we have an isomorphism of graded (resp. filtered) $R$-modules
\[T(W_{\theta}) \simeq \frac{R[A,B,C]}{\big(\text{\eqref{eq:lineareq}, \eqref{eq:quadraticeq}, \eqref{eq:cubiceq}}\big)}(-3),\]
where $deg(A) = deg(B) = deg(C) = 2$.
\item By Proposition \ref{proposition:KhovKuprelgraded}, Example \ref{ex:Tcircle}, and Remark \ref{rem:gradingcircle} we have the following isomorphisms of graded (resp. filtered) $R$-modules
\[T\left(W_{\theta}\right) \simeq T(\bigcirc) (-1) \oplus T(\bigcirc)(1) \simeq \frac{R[x]}{(\omega(x))}(-3) \oplus \frac{R[y]}{(\omega(y))}(-1),\]
where $deg(x) = deg(y) = 2$.
In the follow up, we shall make use of this representation of $T(W_{\theta})$ rather than $M(-3)$. Of course the two representations are isomorphic as graded (resp. filtered) $R$-modules. An explicit isomorphism is given by:
\[ M(-3) \longrightarrow \frac{R[x]}{(\omega(x))}(-3) \oplus \frac{R[y]}{(\omega(y))}(-1)\::\: \left\lbrace\begin{matrix}A^k\phantom{B} \mapsto x^k & k = 0,1,2\\ A^kB \mapsto y^k & k = 0,1,2\end{matrix}\right.\]
\end{enumerate}
\end{rem*}%}

Now, we are ready to prove the following result.

\begin{proposition}\label{proposition:betaiscycle}
Let $D$ be an oriented link diagram. Then, $\beta_{\omega,x_1}(D)$ is a cycle.
\end{proposition}
\begin{proof}
First, notice that the oriented web resolution is bipartite exactly as the oriented resolution; that is, if two arcs in $\underline{w}_{D}$ were connected by a crossing in $D$, then they belong to different circles in $\underline{w}_{D}$.

Let $\underline{w^\prime}$ be a web resolution which is obtained from $\underline{w}_{D}$ by replacing a $0$-web resolution with a $1$-resolution, and denote by $E$ the set of such resolutions. Notice that each $\underline{w^\prime} \in E$ is the disjoint union of circles and a theta web. In particular, we have the following isomorphism of $R$-modules
\[ T(\underline{w^\prime}) \simeq \bigotimes_{\gamma\in \underline{w}_{D}\setminus \{\gamma_{1},\gamma_{2}\}} \frac{R[x_{\gamma}]}{(\omega(x_{\gamma}))} \otimes \left( \frac{R[y]}{(\omega (y))} \oplus\frac{R[x]}{(\omega(x))}\right), \]
where $\gamma_{1} = \gamma_{1}(\underline{w^\prime})$ and $\gamma_{2}= \gamma_{2}(\underline{w^\prime})$ are the two circles in $\underline{w}_{D}$ which are merged into the theta web in $\underline{w^\prime}$, and the circles in $\underline{w^\prime}$ are identified with the corresponding circles in $\underline{w}_{D}$.

By definition, the differential $d_{geo}$ of the geometric complex is of the form
\[ d_{geo\ \vert\:\underline{w}_{D}} = \sum_{\underline{w^\prime}\in E} \pm F(\underline{w^\prime}),\]
where each $F(\underline{w^\prime})$ is a disjoint union of cylinders and a copy of the foam $F$ in Figure \ref{fig:foamlemma}, and the boundary of $F(\underline{w^\prime})$ is $\underline{w^\prime} \sqcup \underline{w}_{D}$.
Applying the tautological functor $T$, we get
\[ d (T(\underline{w}_{D})) \subseteq \bigoplus_{\underline{w^\prime}\in E} \left(\left( \frac{R[y]}{(\omega(y))} \oplus\frac{R[x]}{(\omega(x))}\right)\otimes \bigotimes_{\gamma\in \underline{w}_{D}\setminus \{\gamma_{1}(\underline{w^\prime}),\gamma_{2}(\underline{w^\prime})\}} \frac{R[x_{\gamma}]}{(\omega(x_{\gamma}))}  \right),\]
and
\[d_{\vert\:T(\underline{w}_{D})} = \bigoplus_{\underline{w^\prime}\in E} \pm \left(T(F) \otimes \bigotimes_{\gamma\in \underline{w}_{D}\setminus \{\gamma_{1}(\underline{w^\prime}),\gamma_{2}(\underline{w^\prime})\}} Id_{R[x_{\gamma}]/(\omega(x_{\gamma}))}\right). \]
The statement now follows immediately from Lemma \ref{lemma:dbetavanish}.
\end{proof}

\subsection{The transverse invariance of the $\beta$-chains}

Now, let us analyse the behaviour of the $\beta$-chains under the maps induced by some Reidemeister moves.
These moves include the closures of the mirror images of the transverse Markov moves. During the proofs in this section the homological degree and the quantum degree (or filtration) will be disregarded.  We remark that the chain homotopy equivalences associated to the Reidemeister moves described here are of (filtered) degree $0$ with respect to the quantum grading (resp. filtration) once the appropriate shifts are taken into account.

\subsubsection*{Negative first Reidemeister move}
Let $D$ be an oriented link diagram, and denote by $D_{-}$ the oriented link diagram obtained from $D$ via a negative first Reidemeister move  on a given arc \textbf{a} (Figure \ref{reidmoves12}).

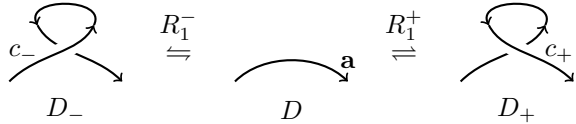
\begin{figure}[H]
\centering
\begin{tikzpicture}[scale = .75, thick]

\draw[->] (-1,0) .. controls +(.5,.5) and +(-.5,.5) ..  (1,0);

\node at (0,-.5) {$D$};
\node at (1,.35) {\textbf{a}};

\node at (4,-.5) {$D_+$};
\node at (-4,-.5) {$D_-$};

\node at (-2,1) {$R_1^-$}; 
\node at (-2,.5) {$\leftrightharpoons$}; 
\node at (2,1) {$R_1^+$}; 
\node at (2,.5) {$\rightleftharpoons$}; 

\draw[<-] (-3,0) .. controls +(-.5,.5) and +(.25,-.5) ..  (-4.5,1);
\pgfsetlinewidth{8*\pgflinewidth}
\draw[white] (-5,0) .. controls +(.5,.5) and +(-.25,-.5) ..  (-3.5,1);
\pgfsetlinewidth{.125*\pgflinewidth}
\draw[<-] (-4.5,1) .. controls +(-.25,.5) and +(.25,.5) ..  (-3.5,1);
\draw[->] (-5,0) .. controls +(.5,.5) and +(-.25,-.5) ..  (-3.5,1);

\begin{scope}[shift= {(8,0)}]
\draw[<-] (-4.5,1) .. controls +(-.25,.5) and +(.25,.5) ..  (-3.5,1);
\draw[->] (-5,0) .. controls +(.5,.5) and +(-.25,-.5) ..  (-3.5,1);
\pgfsetlinewidth{8*\pgflinewidth}
\draw[white] (-3,0) .. controls +(-.5,.5) and +(.25,-.5) ..  (-4.5,1);
\pgfsetlinewidth{.125*\pgflinewidth}
\draw[<-] (-3,0) .. controls +(-.5,.5) and +(.25,-.5) ..  (-4.5,1);
\end{scope}

\node at (4.75,.5) {$c_+$};
\node at (-4.75,.5) {$c_-$};

\end{tikzpicture}
\caption{The negative (left) and positive (right) versions of the first Reidemeister move.}
\label{reidmoves12}
\end{figure} 

In Figure \ref{fig:firstRmove} there is a description of the map associated to a negative Reidemeister move between the geometric complexes (cf. \cite[Section 2.2]{Mackaayvaz07}). The figure should be read as follows: the foams are all embedded in $(\mathbb{R}^{2}\times\{ 0 \}) \times [0,1]$ and are cylinders except in a small cylinder above the arc $\mathbf{a}$, where they look like the ones depicted in Figure \ref{fig:firstRmove}.
\begin{figure}[h]
\begin{tikzpicture}[scale=.3]

%%%%%%%%%%%%%% Resolutions

\draw[dashed] (0,7.5) circle (2) ;
\draw[dashed] (0,-7.5) circle (2) ;
\draw[dashed] (15,7.5) circle (2) ;
\node at (15,-7.5) {0} ;
\draw[thick,<-] (-1.414,7.5 +1.414) .. controls +(.75,-.75) and +(.75,.75) .. (-1.414,7.5-1.414);
\draw[thick,->] (0,7.5) arc (180:-180:.75);
\draw[thick,<-] (15-1.414,7.5 +1.414) -- (15-0.707,7.5 +0.707)-- (15-0.707,7.5 -0.707)-- (15-1.414,7.5-1.414);
\draw[thick,->] (15-0.707,7.5 +0.707) .. controls +(2,2) and +(2,-2) .. (15-0.707,7.5 -0.707);
\draw[very thick, purple, ->] (15-0.707,7.5 +0.707)-- (15-0.707,7.5 -0.707);

\draw[thick,<-] (-1.414,-7.5 +1.414) .. controls +(.75,-.75) and +(.75,.75) .. (-1.414,-7.5-1.414);

%%%%%%%%%%%%%% Arrows=

\draw[thick, ->](4,7.5)--(11,7.5);

\draw[thick, ->](4,-7.5)--(11,-7.5);

\draw[thick, ->](-.5,-4)--(-.5,4);
\draw[thick,<-](.5,-4)--(.5,4);

\draw[thick, ->](15.5,-4)--(15.5,4);
\draw[thick,<-](14.5,-4)--(14.5,4);

%%%%%%%%%%%%%% Cobordisms
\draw (-3.5,0) -- (-3,-1) -- (-3,3)-- (-3.5,4) -- cycle;
\draw (-2.75,-.5) arc (180:0:.75) arc (0:-180:.75 and .25);
\draw[dashed] (-1.25,-.5) arc (0:180:.75 and .25);

\begin{scope}[shift ={(-5.5,0)}]
\draw (-3.5,0) -- (-3,-1) -- (-3,3)-- (-3.5,4) -- cycle;
\draw (-2.75,-.5) arc (180:0:.75) arc (0:-180:.75 and .25);
\draw[dashed] (-1.25,-.5) arc (0:180:.75 and .25);
\end{scope}

\begin{scope}[shift ={(-11.75,0)}]
\draw (-3.5,0) -- (-3,-1) -- (-3,3)-- (-3.5,4) -- cycle;
\draw (-2.75,-.5) arc (180:0:.75) arc (0:-180:.75 and .25);
\draw[dashed] (-1.25,-.5) arc (0:180:.75 and .25);
\end{scope}

\begin{scope}[shift ={(10,13)}, yscale = -1]
\draw (-4,1) -- (-2,-1) -- (-2,3)-- (-4,5) -- cycle; 
\draw[pattern =north east lines, pattern color = red, dashed] (-3.5,4.5) .. controls +(0,-1.5) and +(0,-1.5) .. (-2.5,3.5) -- cycle;
\draw[] (-3,2.85) -- (-2,2);
\draw[] (-3.5,4.5) .. controls +(1.5,0) and +(1.5,0) .. (-2.5,3.5) -- cycle;
\draw[] (-1.25,.5)arc (-180:0:.75 and .25);
\draw[] (.25,.5) arc (0:180:.75 and .25);
\draw (-2,2) .. controls +(.5,-.5) and +(0,.5) .. (-1.25,.5);
\draw (-1.75,3.8) .. controls +(.5,-.5) and +(0,.5) .. (.25,.5);
\draw[thick, red] (-3.5,4.5) .. controls +(0,-1.5) and +(0,-1.5) .. (-2.5,3.5) -- cycle ;
\end{scope}

\begin{scope}[shift = {+(7,-1)}]
\node[left] at (-3.75,2) {$-$};
\draw (-3.5,0) -- (-3,-1) -- (-3,3)-- (-3.5,4) -- cycle;
\draw (-2.75,3.5) arc (-180:0:.75) arc (0:360:.75 and .25);
\draw[dashed] (-1.25,3.5) arc (0:180:.75 and .25);
\end{scope}

\node at (-4.25,1.5) {$a_1$};
\node at (-5.5,1.5) {$-$};
\node at (-11.75,1.5) {$-a_{2}\: \sum_{i=0}^{1}$};
\node at (-17.5,1.5) {$F=\:\sum_{i=0}^{2}\quad $};

\node at (-12.25,-.5) {$i$};
\node at (-13.5,3.5) {$2-i$};

\node at (-6.25,-.5) {$i$};
\node at (-7.5,3.5) {$1-i$};

\node at (-6.25,7.5) {$\langle D_{-}  \rangle_{\omega}\ :$};
\node at (-6.25,-7.5) {$\langle D \rangle_{\omega}\ :$};

\node at (6.5,1) {$=G$};
\node at (0,-7) {$\mathbf{a}$};
\end{tikzpicture}
\caption{Schematic description of the maps encoding a negative first Reidemeister move. The numbers next to the foams (which are drawn in $(\mathbb{R}^{2}\times\{0\})\times [0,1]$ and should be read top to bottom) indicate the number of dots. The horizontal maps are the differentials. Notice that there is a difference in the sign of the maps $F$ and $G$ with respect to \cite{Mackaayvaz07}.}
\label{fig:firstRmove}
\end{figure}
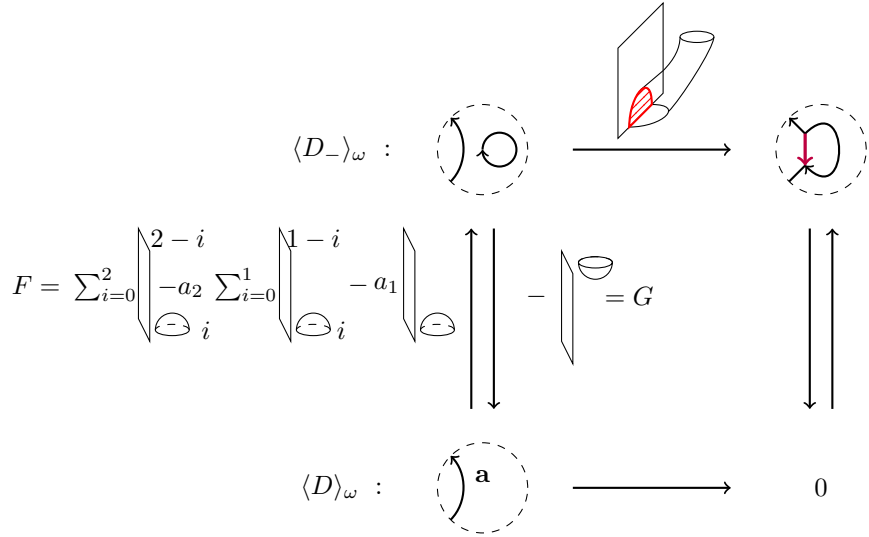

\begin{rem*}
The maps $F$ and $G$ described in Figure \ref{fig:firstRmove} have the opposite sign with respect to the corresponding maps defined in \cite[Section 2.2]{Mackaayvaz07}.
\end{rem*}

Before proceeding we need the following lemma.

\begin{lemma}\label{lemma:techlemmaforfirstmove}
Let $P(x) \in R[x]$ be the polynomial $x^2 + a_{1}^\prime x + a_{0}^\prime$. Then,
\begin{small}
\[[x^2P(x)]\otimes [1] + [xP(x)]\otimes [y] + a_2[xP(x)]\otimes [1] - x_1[P(x)]\otimes[y] -x_1 a_1^\prime [P(x)]\otimes[1]\]
\end{small}
is zero in
\[\frac{R[x]}{(\omega (x))}\otimes \frac{R[y]}{(\omega (y))}\]
\end{lemma}
\begin{proof}
First, let us point out that
\begin{equation}
x P(x) = x_{1} P (x)\quad mod\ \omega(x).
\label{eq:basesl3}
\end{equation}
Using  the equality $a_{2} = a_1^\prime - x_1$, we get
\[[x^2 P(x)]\otimes [1] + [x P(x)]\otimes [y] + a_2[x P(x)]\otimes [1] =\]
\[ = [x^2P (x)]\otimes [1] + [x P (x)]\otimes [y] + a_1^\prime [x P (x)]\otimes [1] - x_1	[xP(x)]\otimes [1] = \]
\[ = x_1[P(x)]\otimes [y] + a_1^\prime x_1[P(x)]\otimes [1], \]
where the last equality follows from Equation \eqref{eq:basesl3}.
\end{proof}

Denote by
\[\Phi_{1}:C_{\omega}^{\bullet}(D)  \longrightarrow C_{\omega}^{\bullet}(D_{-})\]
the map associated to the (linear combination of) foam(s) $F$ in Figure \ref{fig:firstRmove}, and by
\[\Psi_{1}:C_{\omega}^{\bullet}(D_{-})  \longrightarrow C_{\omega}^{\bullet}(D),\]
the map associated to the foam denoted by $G$ in the same figure.

\begin{proposition}\label{proposition:betaandr1sl3}
Let $D$ be an oriented link diagram, and let $D_{-}$ be the diagram obtained from $D$ via a negative first Reidemeister move. Then,
\[\Psi_{1}(\beta_{\omega, x_1}(D_{-})) = \beta_{\omega, x_1}(D)\quad\text{and}\quad \Phi_{1}(\beta_{\omega, x_1}(D)) = \beta_{\omega, x_1}(D_{-}).\]
\end{proposition}
\begin{proof}
Notice that $\underline{w}_{D_{-}}$ is mapped to $\underline{w}_{D}$ by $\Psi_{1}$ and that $\underline{w}_{D_-}$ can be identified with $\underline{w}_{D} \sqcup \bigcirc$. By the Khovanov-Kuperberg circle removal relation, Example \ref{ex:Tcircle} and the sphere relation we can identify $\Psi_{1\: \vert \underline{w}_{D}}$ with the map
\[\frac{R[x_{\gamma^\prime}]}{(\omega (x_\gamma^\prime))} \otimes \bigotimes_{\gamma\in\underline{w}_{D}} \frac{R[x_{\gamma}]}{(\omega (x_\gamma))} \longrightarrow \bigotimes_{\gamma\in\underline{w}_{D}} \frac{R[x_{\gamma}]}{(\omega(x_\gamma))}\]
given by
\[ q_{\gamma^\prime}(x_{\gamma\prime}) \otimes \bigotimes_{\gamma\in\underline{w}_D} q_{\gamma}(x_{\gamma}) \longmapsto -\epsilon(q^\prime(x_{\gamma^\prime}))\bigotimes_{\gamma\in\underline{w}_{D}}  q_{\gamma}(x_{\gamma}), \]
where $\gamma^\prime$ indicates the circle in $\underline{w}_{D_-}\setminus \underline{w}_{D}$ and
\[\epsilon: \frac{R[x]}{(\omega (x))} \longrightarrow R\: :\: [ax^{2} + bx + c] \mapsto -a.\]
Since in the case of $\beta_{\omega, x_1}(D)$ we have $q_\gamma (x) = q_{\gamma^\prime}(x) = [P(x)] = [x^{2} + a_1^\prime x + a_0^\prime]$ the first part of the statement follows.

Similarly to what has been done with $\Psi_1$, we can identify $\Phi_1$ with the map
\[ \bigotimes_{\gamma\in\underline{w}_{D}} \frac{R[x_{\gamma}]}{(\omega (x_\gamma))} \longrightarrow \frac{R[x_{\gamma^\prime}]}{(\omega (x_\gamma^\prime))} \otimes\bigotimes_{\gamma\in\underline{w}_{D}} \frac{R[x_{\gamma}]}{(\omega (x_\gamma))}\]
mapping $ \bigotimes_{\gamma\in\underline{w}_D} q_{\gamma}(x_{\gamma})$ to
\begin{small}
\[-\left(\sum_{i=0}^{2}(x_{\gamma_{\mathbf{a}}}^{2-i}q_{\gamma_{\mathbf{a}}} \otimes x_{\gamma^\prime}^{i})+ a_{2} \sum_{i=0}^{1}(x_{\gamma_{\mathbf{a}}}^{1-i}q_{\gamma_{\mathbf{a}}}\otimes x_{\gamma^\prime}^{i}) + a_{1}(q_{\gamma_{\mathbf{a}}}\otimes 1)\right)\otimes  \bigotimes_{\gamma\in\underline{w}_{D}\setminus \{ \gamma_{\mathbf{a}}\}}  q_{\gamma}(x_{\gamma}), \]
\end{small}
where $\gamma_{\mathbf{a}}$ is the circle in $\underline{w}_{D}$ containing the arc $\mathbf{a}$ (see Figure \ref{fig:firstRmove}) and $q_{\gamma_{\mathbf{a}}} = q_{\gamma_{\mathbf{a}}}(x_{\gamma_{\mathbf{a}}})$.
It is now easy to see that
\begin{small}
\begin{align*}
 \Phi_1 (\beta_{\omega, x_1}(D)) =  &\left( [x_{\gamma_{\mathbf{a}}}^2 P(x_{\gamma_{\mathbf{a}}})]\otimes [1] + [x_{\gamma_{\mathbf{a}}}P(x_{\gamma_{\mathbf{a}}})]\otimes [x_{\gamma^\prime}] + a_2[x_{\gamma_{\mathbf{a}}}P(x_{\gamma_{\mathbf{a}}})]\otimes [1] +\right. \\
& - x_1[P(x_{\gamma_{\mathbf{a}}})]\otimes[x_{\gamma^\prime}]-x_1 a_1 [P(x_{\gamma_{\mathbf{a}}})]\otimes[1] \Big) \otimes  \bigotimes_{\gamma\in\underline{w}_{D}\setminus \{ \gamma_{\mathbf{a}}\}}  P(x_{\gamma}) + \beta_{\omega, x_1}(D_-) =\\
= &\ \beta_{\omega, x_1}(D_-)
\end{align*}
\end{small}
where the last equality is due to Lemma \ref{lemma:techlemmaforfirstmove}.
\end{proof}

\subsubsection*{Second Reidemeister move}

Now, let us turn to the coherent version of the second Reidemeister move. Let $D$ be an oriented link diagram. Let \textbf{a} and \textbf{b} be two (un-knotted) arcs of $D$ lying in a small ball. Performing a second Reidemeister move on these arcs inserts two adjacent crossings, say $c_1$ and $c_2$, of opposite types.

Recall that a Reidemeister move is coherent if it can be obtained by rotating or taking the mirror image of the one in Figure \ref{fig:coherentR22}. Denote by $D^{\prime}$ the link obtained from $D$ by performing a coherent second Reidemeister move. Finally, denote by $\underline{w}^\prime_{D^{\prime}}$ the web resolution of $D^{\prime}$ where all crossings but $c_1$ and $c_{2}$ are resolved as in the oriented web resolution.

\begin{figure}[H]
\centering
\begin{tikzpicture}[scale = .15, thick]%, xscale = -1]

\ocross{-15}{-2}
\icross{-6}{-2}
\draw (-6,2) .. controls +(-1.5,1.5) and +(1.5,1.5) ..  (-11,2); 
\draw (-11,-2) .. controls +(1.5,-1.5) and +(-1.5,-1.5) ..  (-6,-2);

\draw[->] (-15,2) -- (-16,3); 
\draw[->] (-15,-2) -- (-16,-3);
\draw (-2,2) -- (-1,3); 
\draw (-2,-2) -- (-1,-3);

\node at (3,0) {$\rightleftharpoons$};

\draw[->] (21,3) .. controls +(-2,-2) and +(2,-2) ..  (8,3); 
\draw[<-] (8,-3) .. controls +(2,2) and +(-2,2) ..  (21,-3);

\node at (14.5,-5) {$D$};
\node at (-8.5,-5) {$D^{\prime}$};
\node at (-13,3.5) {$c_1$};
\node at (-3.5,3.5) {$c_2$};
\node at (20,3.5) {$\mathbf{a}$};
\node at (20,-3.5) {$\mathbf{b}$};
\end{tikzpicture}
\caption{A coherent version of the second Reidemeister move. All other coherent second Reidemeister moves are obtained by rotating or taking the mirror image of the one in figure.}
\label{fig:coherentR22}
\end{figure}
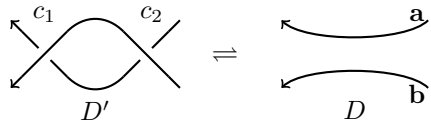

The map associated to the second Reidemeister move at the level of geometric complexes was defined by Mackaay and Vaz as in Figure \ref{fig:secondRmcoer}. Denote by $\Phi_2$ and $\Psi_2$ the two maps
\[\Phi_2 : C_{\omega}^\bullet(D,R)\longrightarrow C_{\omega}^\bullet(D^{\prime},R)\qquad \Psi_2 : C_{\omega}^\bullet(D^{\prime},R)\longrightarrow C_{\omega}^\bullet(D,R)\]
associated to the coherent second Reidemeister move.
\begin{figure}[h]
\centering
\includegraphics[scale=.35]{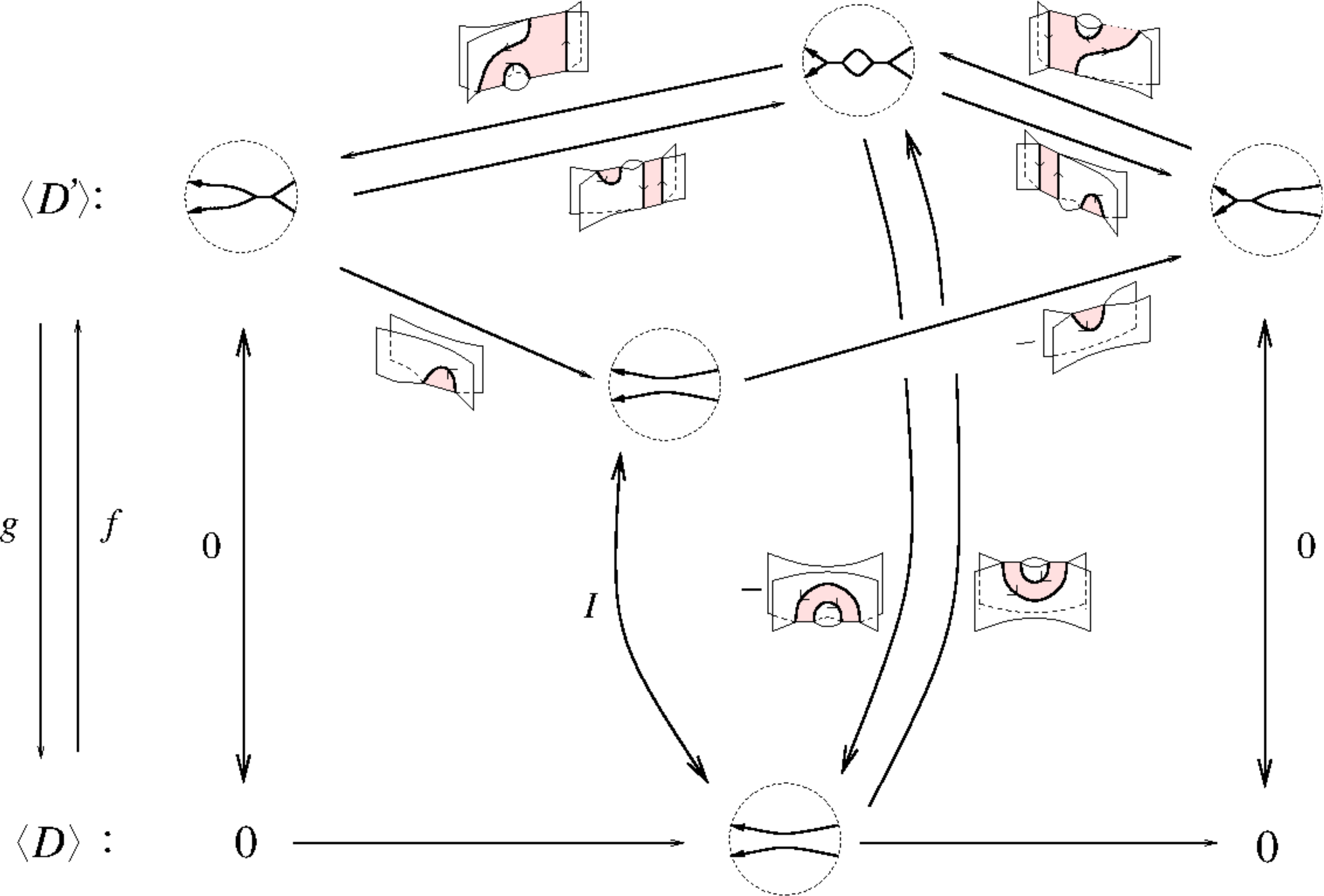} 
\caption{Map associated to the coherent second Reidemeister move.}\label{fig:secondRmcoer}
\end{figure}

\begin{proposition}\label{proposition:secondcoersl3}
Let $D$ be an oriented link diagram and let $D^{\prime}$ be the diagram obtained from $D$ by performing a coherent second Reidemeister move. Then,
\[\Phi_2 (\beta_{\omega, x_1}(D)) = \beta_{\omega, x_1}(D^{\prime})\quad\text{and}\quad\Psi_2 (\beta_{\omega, x_1}(D^{\prime})) = \beta_{\omega, x_1}(D)\]
\end{proposition}
\begin{proof}
First notice that $\underline{w}_{D}$ can be easily identified with $\underline{w}_{D^{\prime}}$. With this identification we have that $\Psi_{2\: \vert T(\underline{w}_{D^\prime})}$ behaves as the identity map (cf. Figure \ref{fig:secondRmcoer}), and the second part of the statement follows.
The map $\Phi_{2\:\vert T(\underline{w}_{D})}$ sends $T(\underline{w}_{D})$ to $T(\underline{w}_{D}) \oplus T(\underline{w}^\prime_{D^{\prime}})$. More precisely, we have
\[ \Phi_{2\:\vert \underline{w}_{D}} = Id_{\underline{w}_{D}} \oplus T(F^\prime),\]
where $F^\prime$ is the foam drawn in Figure \ref{fig:cobFprime}. To conclude it suffices to prove that:
\[T(F^\prime) (\beta_{\omega, x_1}(D)) = 0.\]
This is immediate from Lemma \ref{lemma:dbetavanish}, once one notices that the foam $F^\prime$ is the composition of the foam $F$ in Figure \ref{fig:foamlemma} and a foam $G$ (see Figure \ref{fig:cobFprime}).
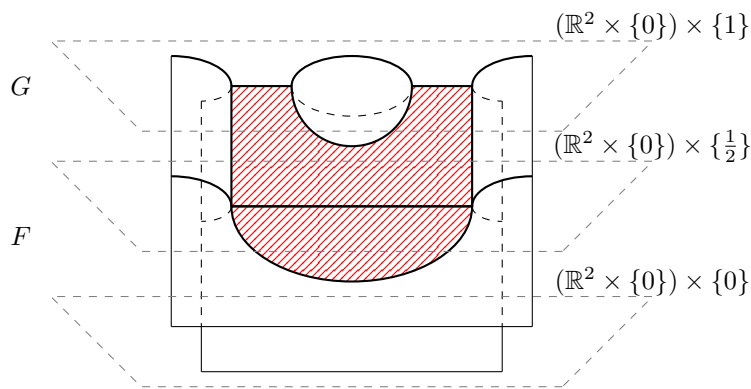
\begin{figure}[h!]
\centering
\begin{tikzpicture}[scale=.4, rotate = 180]
\draw[pattern color = red ,pattern = north east lines, thick] (0,0) arc (0:180:2)--(-6,0) -- (-6,4) -- (2,4) -- (2,0) -- cycle;
\draw[dashed] (0,0) arc (0:180:2 and 1);
\draw[thick] (0,0) arc (0:-180:2 and 1);
\draw[dashed] (-6,0) arc (0:90:1 and .5);
\draw[thick] (-6,0) arc (0:-90:2 and 1);
\draw[dashed] (2,0) arc (180:90:1 and .5);
\draw[thick] (2,0) arc (180:270:2 and 1);
\draw[dashed] (-6,4) arc (0:90:1 and .5);
\draw[thick] (-6,4) arc (0:-90:2 and 1);
\draw[dashed] (2,4) arc (180:90:1 and .5);
\draw[thick] (2,4) arc (180:270:2 and 1);

\draw[pattern color = red ,pattern = north east lines, thick] (2,4) arc (0:180:4 and 2.5) -- cycle;
\draw (-8,8) -- (4,8);
\draw (-7,9.5) -- (3,9.5);
\draw (-8,8) -- (-8,-1);
\draw (4,8) -- (4,-1);
\draw (-7,9.5) -- (-7,8);
\draw (3,9.5) -- (3,8);
\draw[dashed] (-7,.5) -- (-7,8);
\draw[dashed] (3,.5) -- (3,8);

\draw[gray, dashed] (-9,5.5) -- (5,5.5) -- (8,2.5) -- (-12,2.5) -- cycle;
\draw[gray, dashed] (-9,10) -- (5,10) -- (8,7) -- (-12,7) -- cycle;
\draw[gray, dashed] (-9,1.5) -- (5,1.5) -- (8,-1.5) -- (-12,-1.5) -- cycle;

\node at (-12,-2) {$(\mathbb{R}^2\times \{ 0 \})\times \{ 1 \}$};
\node at (-12,2) {$(\mathbb{R}^2\times \{ 0 \})\times \{ \frac{1}{2} \}$};
\node at (-12,6.5) {$(\mathbb{R}^2\times \{ 0 \})\times \{ 0 \}$};

\node at (9,0) {$G$};
\node at (9,5) {$F$};
\end{tikzpicture}
\caption{The cobordism $F^\prime$ as a composition of the cobordism $F$ (on the bottom) and the cobordism $G$ (on the top).}
\label{fig:cobFprime}
\end{figure}
\end{proof}

\subsubsection*{Third Reidemeister move}
Finally, we have to prove the invariance of the $\beta$-chains under braid-like third Reidemeister moves. 
Consider the version $R_{3}^\circ$ of the third Reidemeister move in Figure \ref{fig:thirdOmegatreb}.
\begin{figure}[H]
\centering
\begin{tikzpicture}[thick]
\draw[->] (-.5,.866) -- (.5,-.866);
\pgfsetlinewidth{8*\pgflinewidth}
\draw[white] (-1,0) .. controls +(.5,.75) and +(-.5,.75) .. (1,0);
\pgfsetlinewidth{0.125*\pgflinewidth}
\draw[->] (-1,0) .. controls +(.5,.75) and +(-.5,.75) .. (1,0);
\pgfsetlinewidth{8*\pgflinewidth}
\draw[white] (-.5,-.866) -- (.5,.866);
\pgfsetlinewidth{0.125*\pgflinewidth}
\draw[->] (-.5,-.866)-- (.5,.866);
\node at (1.5,.5){$R_{3}^\circ$};
\node at (1.5,0){$\leftrightharpoons$};
\begin{scope}[shift = {+(3,0)}]
\draw[->] (-.5,.866) -- (.5,-.866);
\pgfsetlinewidth{8*\pgflinewidth}
\draw[white] (-1,0) .. controls +(.5,-.75) and +(-.5,-.75) .. (1,0);
\pgfsetlinewidth{0.125*\pgflinewidth}
\draw[->] (-1,0) .. controls +(.5,-.75) and +(-.5,-.75) .. (1,0);
\pgfsetlinewidth{8*\pgflinewidth}
\draw[white] (-.5,-.866) -- (.5,.866);
\pgfsetlinewidth{0.125*\pgflinewidth}
\draw[->] (-.5,-.866)-- (.5,.866);
\end{scope}
\node at (0,-1.5){$L_1$};
\node at (3,-1.5){$L_2$};
\end{tikzpicture}
\caption{A version of the Third Reidemeister move.}
\label{fig:thirdOmegatreb}
\end{figure}
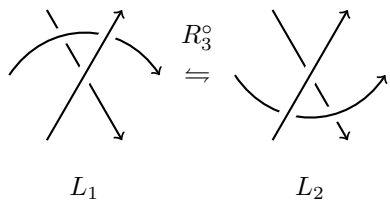
All the braid-like third Reidemeister moves can be deduced, via a sequence of coherent second Reidemeister moves, from the $R_3^\circ$ move (cf. \cite[Lemma 2.6]{Polyak10}). See Figure \ref{fig:fromRplustoRcirc} for an example.
\begin{figure}[H]
\centering
\begin{tikzpicture}[scale =.5, thick]

\draw (1,1) -- (0,1);
\draw (1,0) -- (0,0);
\draw (1,-1) -- (0,-1);

\draw (1,-1) .. controls +(.5,0) and +(-.5,0).. (2,0);
\pgfsetlinewidth{8*\pgflinewidth}
\draw[white] (1,0) .. controls +(.5,0) and +(-.5,0).. (2,-1);
\pgfsetlinewidth{.125*\pgflinewidth}
\draw (1,0) .. controls +(.5,0) and +(-.5,0).. (2,-1);
\draw (1,1) -- (2,1);

\draw (2,1) .. controls +(.5,0) and +(-.5,0).. (3,0);
\pgfsetlinewidth{8*\pgflinewidth}
\draw[white] (2,0) .. controls +(.5,0) and +(-.5,0).. (3,1);
\pgfsetlinewidth{.125*\pgflinewidth}
\draw (2,0) .. controls +(.5,0) and +(-.5,0).. (3,1);
\draw (3,-1) -- (2,-1);

\draw (3,0) .. controls +(.5,0) and +(-.5,0).. (4,-1);
\pgfsetlinewidth{8*\pgflinewidth}
\draw[white] (3,-1) .. controls +(.5,0) and +(-.5,0).. (4,0);
\pgfsetlinewidth{.125*\pgflinewidth}
\draw (3,-1) .. controls +(.5,0) and +(-.5,0).. (4,0);
\draw (3,1) -- (4,1);

\draw (4,1) -- (6,1);
\draw (4,0) -- (6,0);
\draw (4,-1) -- (6,-1);

\begin{scope}[shift = {+(8,0)}]
\draw (0,1) -- (1,1);
\draw (0,0) -- (1,0);
\draw (0,-1) -- (1,-1);

\draw (1,-1) .. controls +(.5,0) and +(-.5,0).. (2,0);
\pgfsetlinewidth{8*\pgflinewidth}
\draw[white] (1,0) .. controls +(.5,0) and +(-.5,0).. (2,-1);
\pgfsetlinewidth{.125*\pgflinewidth}
\draw (1,0) .. controls +(.5,0) and +(-.5,0).. (2,-1);
\draw (1,1) -- (2,1);

\draw (2,1) .. controls +(.5,0) and +(-.5,0).. (3,0);
\pgfsetlinewidth{8*\pgflinewidth}
\draw[white] (2,0) .. controls +(.5,0) and +(-.5,0).. (3,1);
\pgfsetlinewidth{.125*\pgflinewidth}
\draw (2,0) .. controls +(.5,0) and +(-.5,0).. (3,1);
\draw (3,-1) -- (2,-1);

\draw (3,0) .. controls +(.5,0) and +(-.5,0).. (4,-1);
\pgfsetlinewidth{8*\pgflinewidth}
\draw[white] (3,-1) .. controls +(.5,0) and +(-.5,0).. (4,0);
\pgfsetlinewidth{.125*\pgflinewidth}
\draw (3,-1) .. controls +(.5,0) and +(-.5,0).. (4,0);
\draw (4,1) -- (3,1);

\draw (4,1) .. controls +(.5,0) and +(-.5,0).. (5,0);
\pgfsetlinewidth{8*\pgflinewidth}
\draw[white] (4,0) .. controls +(.5,0) and +(-.5,0).. (5,1);
\pgfsetlinewidth{.125*\pgflinewidth}
\draw (4,0) .. controls +(.5,0) and +(-.5,0).. (5,1);
\draw (5,0) .. controls +(.5,0) and +(-.5,0).. (6,1);
\pgfsetlinewidth{8*\pgflinewidth}
\draw[white] (5,1) .. controls +(.5,0) and +(-.5,0).. (6,0); 
\pgfsetlinewidth{.125*\pgflinewidth}
\draw (5,1) .. controls +(.5,0) and +(-.5,0).. (6,0);
\draw (4,-1) -- (6,-1);
\end{scope}

\begin{scope}[shift = {+(8,-4)}]
\draw (5,1) -- (6,1);
\draw (5,0) -- (6,0);
\draw (5,-1) -- (6,-1);

\draw (0,1) -- (1,1);
\draw (4,0) .. controls +(.5,0) and +(-.5,0).. (5,1);
\pgfsetlinewidth{8*\pgflinewidth}
\draw[white] (4,1) .. controls +(.5,0) and +(-.5,0).. (5,0);
\pgfsetlinewidth{.125*\pgflinewidth}
\draw (4,1) .. controls +(.5,0) and +(-.5,0).. (5,0);

\draw (5,-1) -- (4,-1);
\draw (0,-1) .. controls +(.5,0) and +(-.5,0).. (1,0); 
\pgfsetlinewidth{8*\pgflinewidth}
\draw[white] (0,0) .. controls +(.5,0) and +(-.5,0).. (1,-1);
\pgfsetlinewidth{.125*\pgflinewidth}
\draw (0,0) .. controls +(.5,0) and +(-.5,0).. (1,-1);

\draw (1,0) .. controls +(.5,0) and +(-.5,0).. (2,-1);
\pgfsetlinewidth{8*\pgflinewidth}
\draw[white]  (1,-1) .. controls +(.5,0) and +(-.5,0).. (2,0);
\pgfsetlinewidth{.125*\pgflinewidth}
\draw (1,-1) .. controls +(.5,0) and +(-.5,0).. (2,0);
\draw (1,1) -- (2,1);

\draw (2,1) .. controls +(.5,0) and +(-.5,0).. (3,0);
\pgfsetlinewidth{8*\pgflinewidth}
\draw[white] (2,0) .. controls +(.5,0) and +(-.5,0).. (3,1);
\pgfsetlinewidth{.125*\pgflinewidth}
\draw (2,0) .. controls +(.5,0) and +(-.5,0).. (3,1);
\draw (3,-1) -- (2,-1);

\draw (3,0) .. controls +(.5,0) and +(-.5,0).. (4,-1);
\pgfsetlinewidth{8*\pgflinewidth}
\draw[white] (3,-1) .. controls +(.5,0) and +(-.5,0).. (4,0);
\pgfsetlinewidth{.125*\pgflinewidth}
\draw (3,-1) .. controls +(.5,0) and +(-.5,0).. (4,0);
\draw (3,1) -- (4,1);

\end{scope}
\begin{scope}[shift = {+(0,-4)}]
\draw (0,1) -- (1,1);
\draw (5,1) -- (6,1);
\draw (5,0) -- (6,0);
\draw (5,-1) -- (6,-1);

\draw (4,0) .. controls +(.5,0) and +(-.5,0).. (5,1);
\pgfsetlinewidth{8*\pgflinewidth}
\draw[white] (4,1) .. controls +(.5,0) and +(-.5,0).. (5,0);
\pgfsetlinewidth{.125*\pgflinewidth}
\draw (4,1) .. controls +(.5,0) and +(-.5,0).. (5,0);
\draw (5,-1) -- (4,-1);
\draw (0,0) --  (2,0);
\draw (0,-1) --(2,-1);

\draw (1,1) -- (2,1);

\draw (2,1) .. controls +(.5,0) and +(-.5,0).. (3,0);
\pgfsetlinewidth{8*\pgflinewidth}
\draw[white]   (2,0) .. controls +(.5,0) and +(-.5,0).. (3,1);
\pgfsetlinewidth{.125*\pgflinewidth}
\draw (2,0) .. controls +(.5,0) and +(-.5,0).. (3,1);
\draw (3,-1) -- (2,-1);

\draw (3,0) .. controls +(.5,0) and +(-.5,0).. (4,-1);
\pgfsetlinewidth{8*\pgflinewidth}
\draw[white] (3,-1) .. controls +(.5,0) and +(-.5,0).. (4,0);
\pgfsetlinewidth{.125*\pgflinewidth}
\draw (3,-1) .. controls +(.5,0) and +(-.5,0).. (4,0);
\draw (3,1) -- (4,1);

\end{scope}
\draw[dashed] (-.25,-5.25) rectangle (6.25,1.25);
\node at (3,-2) {$\downharpoonleft\hspace{-2.5pt} \upharpoonright$};
\node at (7,0) {$\leftrightharpoons$};
\node at (7,-4) {$\leftrightharpoons$};
\node at (11,-2) {$\downharpoonleft\hspace{-2.5pt} \upharpoonright$};
\end{tikzpicture}
\caption{How to recover another braid like third Reidemeister move (in the dashed box) with an $R_3^\circ$ and two coherent second Reidemeister moves.}
\label{fig:fromRplustoRcirc}
\end{figure}
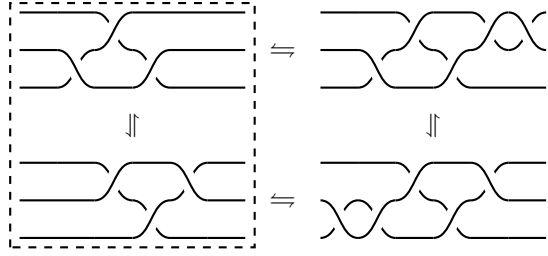

Chain maps between the geometric complexes associated to $R_3^{\circ}$ have been described explicitly by Mackaay and Vaz. Each of these maps is defined as the composition of two maps. First, one defines an element $Q \in \mathbf{Kom}(\mathbf{Foam}_{/\ell})$ which is \emph{not} the geometric complex associated to a link. Then, one defines chain maps
\[ F_i : \langle D_i \rangle_{\omega} \longrightarrow Q\qquad G_{i}: Q \longrightarrow \langle D_i \rangle_{\omega},\]
where $ i\in \{1,2\}$, and  $D_1$ and $D_2$ are the diagrams on each side of the $R_3^\circ$ move (Figure \ref{fig:thirdOmegatreb}), such that $F_i$ is the up-to-homotopy inverse of $G_i$. For a description of such maps the reader may refer to Figure \ref{fig:invterzasl3a} (cf. \cite{Mackaayvaz07}).

The important thing that the reader should keep in mind is that for each web resolution $\underline{w}$ of $D_i$ such that the crossings involved in $R_3^\circ$ are resolved as in the oriented resolution, there is the same direct summand in $Q$, and that the restriction of either $G_i$ or $F_i$ to $\underline{w}$ is minus the identity cobordism.

Finally, the maps associated to each direction of $R_3^\circ$ (between the Kuperberg brackets) 
\[\widetilde{\Psi}_{3}: \langle D_1 \rangle_{\omega} \longrightarrow \langle D_2 \rangle_{\omega} \qquad \widetilde{\Phi}_{3}: \langle D_2  \rangle_{\omega} \longrightarrow \langle D_1   \rangle_{\omega}\]
are defined as follows
\[ \widetilde{\Psi}_3 = G_2 \circ F_1 \qquad \widetilde{\Phi}_3 = G_1 \circ F_2.\]
It is immediate that these two maps, when restricted to the oriented web resolutions, are cylinders. Denote by $\Psi_3$ and $\Phi_3$ the maps between $\mathfrak{sl}_3-$complexes associated to $\widetilde{\Psi}_3$ and $\widetilde{\Phi}_3$, respectively. Then, maps $\Psi_3$ and $\Phi_3$ behave as the identity maps between the summands associated to the oriented web resolutions. So the following proposition is immediate.

\begin{proposition}
Let $D_1$ and $D_2$ be two oriented link diagrams related by a coherent third Reidemeister move. Then,
\[
\Psi_3 (\beta_{\omega,\: x_1}(D_1)) =\beta_{\omega,\: x_1}(D_2)\quad\text{and}\quad \Phi_3 (\beta_{\omega,\: x_1}(D_2)) = \beta_{\omega,\: x_1}(D_1).
\]\qed
\end{proposition}

\begin{proof}[Proof of  Theorem \ref{theorem:beta_3}]
Theorem \ref{theorem:beta_3} follows immediately by putting together the results concerning the behaviour of $\beta_{\omega,\: x_1}$ under the maps induced by coherent Reidemeister moves and negative first Reidemeister moves, and the computation of the degrees after the definition of $\beta_{\omega,x_{i}}$.
\end{proof}

We will call $\beta_{\omega,\: x_1}(\overline{B})$ the \emph{$\beta_3$-invariant of $B$ associated to $(\omega, x_1)$}.

\begin{rem*}
The $\psi_3$-invariant introduced by Wu in \cite{Wu08} is a special case of our construction, more precisely  $\psi_3$ is the homology class of the $\beta_3$-invariant associated to $(x^3, 0)$.
\end{rem*}

\section{Auxiliary invariants}

In the previous section, we already proved the existence of transverse invariants in the $\mathfrak{sl}_3$-chain complex obtained from a factorisable potential (i.e. a potential admitting at least a root in the base ring). At this point a natural question arises: how much information do these invariants contain? That is, are these invariants effective? Unfortunately we do not have an explicit answer to this question. This is partly due to the lack of sufficiently simple (non-trivial) examples on which to perform the computations, and also due to the fact that these invariants are, in some sense, difficult to handle. More precisely, they are chains in a diagram-dependent chain complex, and to prove that two braids have distinct invariants one has to prove that there \emph{does not exist} a map (induced by sequences of Reidemeister/Markov moves between the closures of the two braids) sending the $\beta_3$-invariants of one braid to the corresponding $\beta_3$-invariants of the other braid. Proving this is not easy in general. The most natural thing to do in this setting is to prove that the homology classes of the $\beta_3$-invariants associated to the two braids behave differently with respect to a given structure on  the $\mathfrak{sl}_3$-homology, which is preserved by the maps induced by sequences of Reidemeister/Markov moves. In this section we shall present some ways to extract information from the homology classes of the $\beta_3$-invariants making use of the $R$-module structure of the $\mathfrak{sl}_3$-homology.

\subsection{The vanishing of the homology class}

The most basic structure on  the $\mathfrak{sl}_3$-homology is that of an $R$-module.
The simplest way to make use of this structure is to look at the vanishing of the homology class of the $\beta_{3}$-invariants. That is, if one braid has a $\beta_3$-invariant whose homology class vanishes while the other braid does not, then the two braids represent distinct transverse links.

For the sake of simplicity, let us assume throughout this section the potential to be completely factorisable in $R$ (i.e. all its roots are in $R$), and $R = \mathbb{F}$ to be a field. Under these hypotheses we can prove Proposition \ref{prop:vanishing}.

\begin{proof}[Proof of Proposition \ref{prop:vanishing}]
Before going into the details of the proof, which is quite long, it is worth to schematically describe the idea behind it. We wish to make use of the Mackaay-Vaz classification of the isomorphism types of $H_{\omega}^{\bullet}(L,\mathbb{F})$ depending on the multiplicity of the roots of $\omega$ (see \cite{Mackaayvaz07}). This classification, as well as its generalisation due to Rose and Wedrich in \cite{RoseWedrich16}, is some sort of generalisation of the Chinese remainder theorem (which can be thought as the case $L= \bigcirc$). In each of the three cases we shall identify the image of the $\beta$-cycles under the isomorphism which describes the isomorphism type of $H^\bullet_{\omega}$. This identification will give us the desired result.

If the potential admits distinct roots in $\mathbb{F}$, then the homology classes of the corresponding $\beta_3$-invariants are linearly independent by an argument essentially due to Gornik (cf. \cite{Gornik} and \cite[Section 3]{Mackaayvaz07}). More precisely, in this case the $\beta$-cycles are a rescaling of some of the so-called ``canonical generators''. Since the independence of the homology classes of the ``canonical generators'' was proved in \cite{Gornik} and \cite{Mackaayvaz07}, the claim follows in this case. 

\begin{rem*}\label{rem:gornillobbmackay}
The proofs in \cite{Gornik} and \cite{Mackaayvaz07} only concern the case $\mathbb{F} = \mathbb{C}$. However, the proof of the independence works over any integral domain $R$, provided that $\omega$ has all roots in $R$. 
\end{rem*}

Let us shift to the case when potential $\omega$ has a double root $x_1$, and a simple root $x_2$. That is $x_{1} = x_3$. In \cite[Theorem 3.18]{Mackaayvaz07} it was proved that
\begin{equation}
\label{eq:decomposition_of_Homega}
 H_{\omega}^{i}(L,\mathbb{F}) \cong \bigoplus_{L^\prime \subseteq L} Kh^{-i-lk(L^\prime,\: L\setminus L^\prime)}(\overline{L^\prime},\mathbb{F}),
\end{equation}
where $L^\prime \subseteq L$ means that $L^\prime$ ranges among the sub-links of $L$, and $Kh$ denotes the original ($\mathfrak{sl}_2$-)Khovanov homology. 
\begin{rem*}
Some remarks are in order: 
\begin{enumerate}[(a)]
\item the empty sub-link and $L$ are also counted among the sub-links;
\item in \cite{Mackaayvaz07} the isomorphism in \eqref{eq:decomposition_of_Homega} is proved only for $\mathbb{F}=\mathbb{C}$. The proof works without change for any field, provided that $\omega$ has all roots in $\mathbb{F}$. 
\item the isomorphism above is \emph{not} canonical but depends on a number of choices;
\item the isomorphism in \eqref{eq:decomposition_of_Homega} is a slight rephrasing of the statement of \cite[Theorem 3.18]{Mackaayvaz07}. More precisely, the statement of \cite[Theorem 3.18]{Mackaayvaz07} reads
\[ U_{a,b,c}^{i}(L,\mathbb{F}) \cong \bigoplus_{L^\prime \subseteq L} KH^{i-lk(L^\prime,\: L\setminus L^\prime)}(L^\prime,\mathbb{F}),\]
where $U_{a,b,c}^{i}$ denotes $H_{\omega}^{i}$ and $KH$ denotes a theory which is equivalent to Khovanov homology. To be precise, $KH$ denotes the Khovanov-Rozansky $\mathfrak{sl}_2$-homology with the homological degree reversed, that is $KH^{i,j}(K,\mathbb{F}) = Kh^{-i,j}(\overline{K},\mathbb{F})$. To see this, the reader can compare \eqref{eq:decomposition_of_Homega} and \cite[Theorem 3.18]{Mackaayvaz07} with the analogous (more general) result \cite[Theorem 1]{RoseWedrich16} and subsequent examples (cf. the conventions in \cite{Khovanov03, KhovanovRozansky05a} and the results in \cite{Mackaayvaz07a}).
\end{enumerate}
\end{rem*}

Items (1) and (2) in the statement shall follow from a careful inspection of the proof of \cite[Theorem 3.18]{Mackaayvaz07}. In order to keep the paper as self-contained as possible, we shall review the crucial steps of the proof.

Let $D$ be an oriented diagram representing an oriented link $L$. A \emph{colouring of} $D$ is a function $\phi$ associating to each arc of $D$ (seen as a graph) a root of $\omega$. A colouring $\phi$ is \emph{compatible} with a web resolution $\underline{w}$ if we can colour all the thick edges of $\underline{w}$ in such a way that at each vertex the set of colours is the set of roots of $\omega$. 
Denote by $\mathbb{W}_{\phi}(D)$ the set of resolutions compatible with a colouring $\phi$.

\begin{rem*}
Given a web resolution $\underline{w}$ compatible with $\phi$, the colour of each thick edge is uniquely determined.
\end{rem*}

Mackaay and Vaz proved that one can associate a sub-complex $C_{\omega}^{\bullet}(\phi ,\mathbb{F})$ to each colouring $\phi$. Furthermore, the complex $C_{\omega}^\bullet (D,\mathbb{F})$ decomposes as the direct sum of these complexes. It turns out that the homology of $C_{\omega}^{\bullet}(\phi ,\mathbb{F})$ is trivial unless $\phi$ assigns the same colour to all arcs belonging to the same components (\emph{proper colouring}). Finally, one proves that for such colourings
\[ H_{\omega}^{\bullet}(\phi ,\mathbb{F}) \cong Kh^{-\bullet-lk(L_\phi,\: L\setminus L_\phi)}(\overline{L_{\phi}},\mathbb{F}),\]
where $L_\phi$ is the sub-link of $L$ corresponding to the sub-diagram of $D$ obtaned by deleting all arcs whose colour with respect to $\phi$ is $x_1$. Furthermore, to each $\underline{w}\in\mathbb{W}_{\phi}(D)$ and $\phi$ proper colouring, it is possible to associate a cycle $\Sigma_{\phi}(\underline{w})\in T(\underline{w})$ called \emph{canonical generator}\footnote{Even though they are not canonically defined!}. For each proper colouring $\phi$ the set $\{[\Sigma_{\phi}(\underline{w})]\}_{\underline{w}\in\mathbb{W}_{\phi}(D)}$ generates the $\mathbb{F}$-vector space $H_{\omega}^{\bullet}(\phi ,\mathbb{F})$.

Denote by $\phi_{i}$ the colouring of $D$ where all arcs are coloured with $x_i$, for $i\in \{1,2\}$. 
The sub-links associated to $\phi_1$ and $\phi_2$ are the whole link $L$ and the empty sub-link $\emptyset$, respectively. 
The sub-modules generated by $\beta_{\omega,x_{1}}(D)$ and $\beta_{\omega,x_{2}}(D)$ are contained in $C_{\omega}^\bullet (\phi_1,\mathbb{F})$ and $C_{\omega}^\bullet (\phi_2,\mathbb{F})$ respectively. Thus, the homology classes of $\beta_{\omega,x_{1}}(D)$ and $\beta_{\omega,x_{2}}(D)$ are mapped to $Kh^{0}(\overline{L},\mathbb{F})$ and $Kh^{0}(\emptyset,\mathbb{F}) = \mathbb{F}$, respectively, by the isomorphism in \eqref{eq:decomposition_of_Homega}.

\begin{rem*}
The above reasoning proves, in particular, that the homology classes of $\beta$-invariants corresponding to different roots (if non-vanishing) are always linearly independent.
\end{rem*}

By inspecting the argument in \cite{Mackaayvaz07}, it is not difficult to see that the image of the homology class of $\beta_{\omega,x_{2}}(D)$ is non-trivial. It is immediate from the definition of the canonical generators that $\beta_{\omega,x_{2}}(D)$ is $(x_1 - x_2)^r\Sigma_{\phi_2}(\underline{w}_{D})$ for a certain $r$, where $\underline{w}_D$ denotes the oriented web resolution of $D$. Thus, the homology class of $\beta_{\omega,x_{2}}(D)$ generates $H_{\omega}^{\bullet}(\phi_2,\mathbb{F}) \cong \mathbb{F}$, and this concludes the proof of item (1).

\begin{rem*}
An alternative proof of the non-vanishing of $[\beta_{\omega,x_{2}}(D)]$ can be found in \cite[Proposition 1.3]{LewarkLobb18}; the cycle ``$\psi(D)$''\footnote{Not to be confused with Plamenevskaya's $\psi$-invariant.} defined in \cite{LewarkLobb18} (for $n=3$) coincide with $\beta_{x^3 - x^2,1}(D)$. Moreover, the same argument used in \cite{LewarkLobb18} to prove that $[\beta_{x^3 - x^2,1}(D)]$ is always non-trivial, can be adapted to the case of an arbitrary (degree $3$) potential with a double and a single root (cf. third paragraph in \cite[Section 4]{LewarkLobb18}).
\end{rem*}

It takes a bit more care to prove that the image of $[\beta_{\omega,x_1}(D)]$ is (a non-zero multiple of) the Plamenevskaya invariant $\psi$. First notice that $\beta_{\omega,x_1}(B)$ is $P\cdot \Sigma_{\phi_1}(\underline{w}_{D})$ where
\[P = \prod_{\gamma \in \underline{w}_{D}} (x_2 - x_1)(X_{\gamma} - x_1) \in R_{\phi_{1}}(\underline{w}_{D}) = \frac{\mathbb{F}[X_{\gamma}\: \vert\: \gamma \in \underline{w}_{D}]}{\langle \omega(X_{\gamma})\: \vert\: \gamma \in \underline{w}_{D} \rangle},\] 
and $X_{\gamma}$ acts on $T(\underline{w}_{D})$ by adding a dot on the regular region bounding $\gamma$. This is implied by the equality
\[ (X-x_1)\left[ (x_2 - x_1)^2 - (X-x_1)^2 \right] \equiv (x_2 - x_1)(X-x_1)(X - x_2)\qquad \mathrm{mod} \:\: \omega(X)\]
which is easily verified. Then, the chain map in \cite[Equation (17)]{Mackaayvaz07} which defines the isomorphism between $H_{\omega}^{\bullet}(\phi ,\mathbb{F})$ and $ Kh^{-\bullet-lk(L_\phi,\: L\setminus L_\phi)}(\overline{L_{\phi}},\mathbb{F})$, behaves as follows
\[\beta_{\omega,x_1}(B) = P\cdot \Sigma_{\phi_1}(\underline{w}_{D}) \mapsto (x_2 - x_1)^{r} X  \otimes \cdots \otimes X\quad \text{for some }r\in \mathbb{N},  \]
where $X \otimes \cdots \otimes X =\widetilde{\psi}$ belongs to the summand associated to the oriented resolution in the Khovanov chain complex. Since,  $ [\widetilde{\psi}] = \psi$, item (2) follows.

Finally, assume the potential $\omega$ has a triple root. Consider the endofunctor $\Phi$ of $\mathbf{Foam}$ described in Figure \ref{fig:endofunctor}. $\Phi$ translates the local relations $\ell_1$ associated to the potential $(x-x_1)^3$ into the local relations $\ell_0$ associated to the potential $x^3$. It follows that $\Phi$ induces an equivalence of $\mathbb{F}$-linear categories between $\mathbf{Foam}_{/\ell_1}$ and $\mathbf{Foam}_{/\ell_0}$. This equivalence of categories induces an isomorphism between the two complexes $C_{\omega}^{\bullet}$ and $C_{x^{3}}^{\bullet}$.
\begin{figure}
\centering
\begin{tikzpicture}[scale =.75]
\draw[dashed] (0,0) circle (1);
\draw[fill] (0,0) circle (.1);
\node at (2,0) {$\longmapsto$};
\draw[dashed] (4,0) circle (1);
\draw[fill] (4,0) circle (.1);
\node at (5.5,0) {$-$};
\node at (6,0) {$x_1$};
\draw[dashed] (7.5,0) circle (1);
\end{tikzpicture}
\caption{}\label{fig:endofunctor}
\end{figure}
Finally, it is immediate to see that this isomorphism sends the $\beta_3$-invariant to the $\psi_3$-invariant. 
\end{proof}

Proposition \ref{prop:vanishing}, together with the results in \cite{TransFromKhType17}, implies the following.

\begin{corollary}
Let $\mathbb{F}$ be a field of characteristic different from $2$. If $\omega$ has a double root $x_{1}\in \mathbb{F}$, then the vanishing of $[\beta_{\omega,x_{1}}(\overline{B})]$ is a non-effective invariant for all $B$'s representing a knot with less than 12 crossings.\qed
\end{corollary}

\subsection{Divisibility, numerical invariants and Bennequin inequalities}\label{sec:div}

There are also other ways to make use of the $R$-module structure of the $\mathfrak{sl}_3$-homology. Let us recall the definition of the $c$-invariants. Let $R$ be an integral domain and $a\in R\setminus\{ 0\}$ a non-unit element. Given a potential $\omega$ and a root $x_1\in R$, the number
\[c_{\omega, x_1}(B;a) = max\left\{ k\:\vert\: \exists\: [y]\in H^{0}_{\omega}(\overline{B})\:\text{such that}\: a^k[y] = [\beta_{\omega, x_1}(\overline{B})]\right\}\in \mathbb{N}\cup \{ \infty \}\]
is a well-defined transverse invariant, where $c_{\omega,x_1}(B;a)=\infty$ if and only if $[\beta_{\omega, x_1}(\overline{B})]$ is trivial or $a$-torsion. These invariants are particularly useful in the case $\omega$ has only single roots: in this case the homology classes of the $\beta_3$-invariants are non-trivial and non-torsion (cf. Remark \ref{rem:gornillobbmackay}). Now assume
\begin{equation} 
R = \mathbb{F}[U]\quad \text{and}\quad \omega(x) = (x- U x_1)(x- U x_2)(x- Ux_3),
\label{eq:setting}
\end{equation}
where $x_{i} \in \mathbb{F}$, for all $i$, and $x_{i} \ne x_{j}$, if $i\ne j$.
By setting $deg(U) =2$ the theory becomes graded, and we also have the following exact sequences of complexes of $\mathbb{F}$-vector spaces
\begin{equation}
0 \longrightarrow C_{\omega}^{\bullet,q}(\widehat{B},\mathbb{F}[U]) \overset{U\cdot}{\longrightarrow} C_{\omega}^{\bullet,q}(\widehat{B},\mathbb{F}[U]) \overset{\pi_0}{\longrightarrow} C_{x^{3}}^{\bullet,q}(\widehat{B},\mathbb{F}) \to 0
\label{eq:exactsequence1}
\end{equation}
for each $q\in \mathbb{Z}$, and
\begin{equation}
0 \longrightarrow C_{\omega}^{\bullet}(\widehat{B},\mathbb{F}[U]) \overset{(U-1)\cdot}{\longrightarrow} C_{\omega}^{\bullet}(\widehat{B},\mathbb{F}[U]) \overset{\pi_1}{\longrightarrow} C_{\omega_{\vert U = 1}}^{\bullet}(\widehat{B},\mathbb{F}) \to 0.
\label{eq:exactsequence2}
\end{equation}
Moreover, it is immediate that
\[\pi_{0}(\beta_{\omega,x_{i}}(\overline{B})) = \beta_{x^{3},0}(\overline{B})\qquad\qquad \pi_{1}(\beta_{\omega,x_{i}}(\overline{B})) = \beta_{\omega_{\vert U = 1},x_{i}}(\overline{B}).\]
The following proposition follows immediately from \eqref{eq:exactsequence1}.
\begin{proposition}
Let $R$, $\omega$ and $x_{i}$ be as above, then the following are equivalent
\begin{enumerate}
\item $\psi_{3}(B) \ne 0$;
\item $c_{\omega, Ux_{i}}(B;U) = 0$ for any choice of $\omega$ and $x_{i}$;
\item $c_{\omega, Ux_{i}}(B;U) = 0$ for all choices of $\omega$ and $x_{i}$; 
\end{enumerate}
for each braid $B$.\qed
\end{proposition}

Now let us recall a few facts about concordance invariants defined from $\mathfrak{sl}_{3}$-link homologies. These invariants were defined in the more general setting of the deformations of $\mathfrak{sl}_{n}$ Khovanov-Rozansky (KR) homologies. However, the statements here shall be restricted to the case $n=3$. It is worth noticing that the deformation of $\mathfrak{sl}_{3}$ KR homology corresponding to the potential $\omega$ coincides with the theory defined in this paper (cf. \cite{Mackaayvaz07a}).

\begin{theorem}[Theorem 1.1, \cite{LewarkLobb17}]\label{thm:jis}
Let $K$ be a knot. Given a potential $\omega\in\mathbb{F}[x]$ with distinct roots $x_1$, $x_2$ and $x_3\in \mathbb{F}$, then we have the following isomorphism of bi-graded vector spaces;
\[Gr^{*}H_{\omega}^{\bullet}(K,\mathbb{F}) \simeq \bigoplus_{i=1}^{3} \mathbb{F}(0,j_{i})\qquad\text{with } j_1 \leq j_2 \leq j_3,\]
where $Gr^*H_{\omega}^\bullet$ is the associated graded object of $H_{\omega}^{\bullet}(K,\mathbb{F})$ (endowed with the quantum filtration), and $\mathbb{F}(h,q)$ is a copy of $\mathbb{F}$ generated in bi-degree $(h,q)$. Furthermore, the $j_{i}$'s are concordance invariants which provide lower bounds to the value of the slice genus.\qed
\end{theorem}

Lewark and Lobb in \cite{LewarkLobb17} defined two other concordance invariants. The first one is just a rescaled average of the $j_i$'s, that is
\[ s_{\omega}(K) = \frac{j_1(K) + j_2(K) +j_3(K)}{12}\in \frac{1}{4}\mathbb{Z},\]
which is a concordance quasi-homomorphism. The second invariant depends on the choice of a root $x_i$ and is denoted by $\tilde{s}_{\omega,x_{i}}(K)$. This is a slice-torus knot invariant (cf. \cite{LivingstonNaik, Lewark14}), and in particular a concordance homomorphism. Since its definition involves constructions which we do not wish to introduce, we refer the reader to \cite{LewarkLobb17} for it. All these invariants are somehow related, as is stated in the following proposition.
\begin{proposition}[Proposition 2.12, \cite{LewarkLobb17}]
Let $K$ be a knot. Given a potential $\omega\in\mathbb{F}[x]$ with distinct roots $x_1$, $x_2$ and $x_3\in \mathbb{F}$, order the roots in such a way that
\[ \tilde{s}_{\omega,x_1}(K) \leq\tilde{s}_{\omega,x_2}(K) \leq\tilde{s}_{\omega,x_3}(K). \]
Then, the following inequality holds
\[ \vert j_{i} - 6 \tilde{s}_{\omega,x_i}(K)\vert \leq 2\]
for each $i \in \{ 1,2,3\}$.\qed
\end{proposition}

To conclude this parenthesis we wish to point out that: $-j_{i}(\overline{K}) = j_{i}(K)$ (\cite[Proposition 2.13]{LewarkLobb17}). It follows that $s_{\omega}(K) = - s_{\omega}(\overline{K})$. 
Now we are ready to prove Proposition \ref{prop:Bennequin-type}.

\begin{proof}[Proof of Proposition \ref{prop:Bennequin-type}]
Directly from the definition of $s_{\omega_{i}}$ follows that Equation \eqref{eq:inequality1} implies Equation \eqref{eq:inequality2}. Furthermore, from \cite[Proposition 2.12]{LewarkLobb17} it is immediate that Equation \eqref{eq:inequality1} implies Equation \eqref{eq:inequality3}.
So it is sufficient to prove \eqref{eq:inequality1}. We borrow the notation from the proof of Proposition \ref{prop:vanishing}.
Denote by $[y]$ an homogeneous element of $ H_{\omega}^{0}(K,\mathbb{F})$ such that
\[ U^{c_{\omega,x_{i}}(B,U)}[y] = [\beta_{\omega,x_{i}}(B)].\]
It is immediate that $\pi_{1}([y])= \pi([\beta_{\omega,Ux_{i}}(B)])$, and the latter is a non-trivial multiple of the homology class of the canonical generator $\Sigma_{\phi_{i}}(\underline{w}_{D})$. Since, the quantum filtration is increasing it follows that
\[ -2sl(B) - 2 c_{\omega,Ux_{i}}(B,U) = qdeg([y]) \geq Fdeg([\Sigma_{\phi_{i}}(\underline{w}_{D})]) =: q_{i}(\overline{K}).\]
It is easy to prove that the maximum filtered degree of the elements of a basis of a vector space equipped with an increasing filtration  does not depend on the chosen basis. Thus, we obtain that
\[-j_{1}(K) = j_{3}(\overline{K}) = \underset{i}{max}\:\{ q_{i}(\overline{K}) \}\leq   -2sl(B) - 2 c_{\omega,Ux_{i}}(B,U).\]
Using the isomorphism induced by the endofunctor of $\mathbf{Foam}$ obtained from the one described in Figure \ref{fig:endofunctor} by replacing $x_1$ with $x_{i} - x_{j}$, it follows that $q_{i} = q_{j}$; and this concludes the proof.
\end{proof}

\begin{corollary}
Let $K$ be a knot and $B$ a braid representing $K$. If either $j_{1}(K) = sl(B)$, $s_{\omega_{1}}(K) = sl(B)$ or $3 \tilde{s}_{\omega,x_{i}}(K) = sl(B)-1$ then $\psi_{3} (B) \neq 0$. In particular, if $B$ is a quasi-positive braid $\psi_{3} (B) \neq 0$.
\end{corollary}
\begin{proof}
The only thing to prove is that quasi-positive braids are such that $3 \tilde{s}_{\omega,x_{i}}(K) = sl(B)-1$. This is true because $3\tilde{s}_{\omega,x_{i}}(K)$ is a slice torus invariant, and thus its value on quasi-positive braids is exactly $sl(B)-1$ (see \cite{Lewark14}).
\end{proof}

\bibliography{Bibliography.bib}
\bibliographystyle{plain}
\begin{landscape}
\begin{figure}[]
\begin{turn}{90}
\includegraphics[scale=.35]{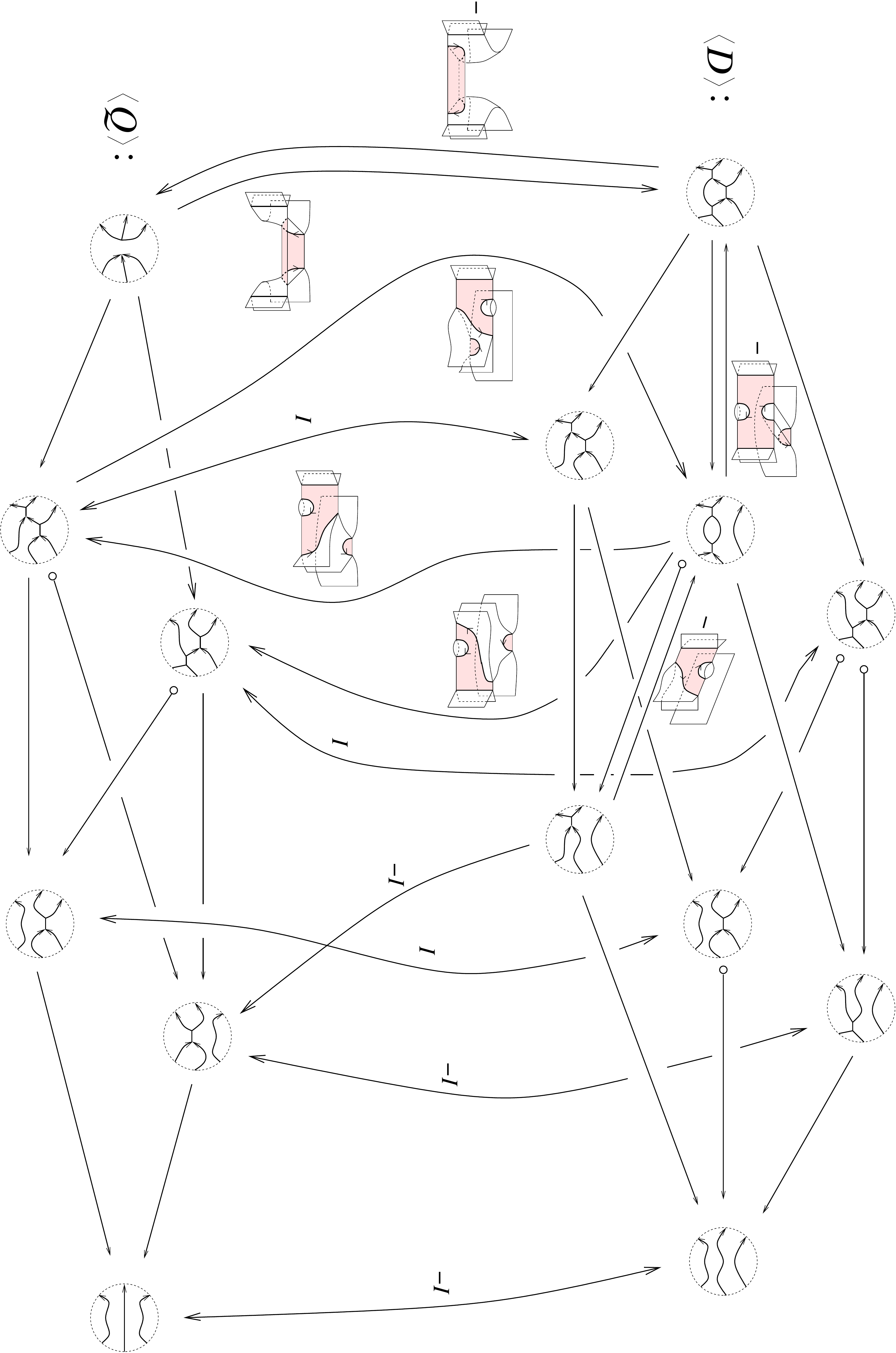} 
\end{turn}
\caption{The maps $G_1$ and $F_1$. The maps $G_2$ and $F_2$ are obtained by reflecting the local figures along the horizontal axes, and the cobordisms accordingly. The red arrows indicate the maps which get a minus sign and the I's indicate the identity we cobordisms.}
\label{fig:invterzasl3a}
\end{figure}
\end{landscape}
\end{document}